\documentclass[a4paper,11pt]{amsart}
\usepackage[utf8]{inputenc}
\usepackage[english]{babel}
\usepackage{fancyhdr}
\usepackage{amsmath}
\usepackage{amsfonts,amstext,mathrsfs}
\usepackage{amssymb,amsthm,amscd,amsxtra,wasysym,graphicx}
\usepackage{mathrsfs,esint,comment}
\usepackage{tikz-cd}
\usepackage{enumitem,mathtools,dsfont}
\usepackage{palatino,mathpazo}
\usepackage{xcolor}

\def\Bibtex{{\rm B\kern-.05em{\sc i\kern-.025em b}\kern-0.08em T\kern-.1667em\lower.7ex\hbox{E}\kern-.125emX}}
\usepackage{hyperref}
\hypersetup{
     unicode=false,          
    pdftoolbar=true,       
    pdfmenubar=true,       
    pdffitwindow=false,     
    pdfstartview={FitH}, 
    pdftitle={Monge-Amp\`ere equations},   
    pdfauthor={Quang-Tuan Dang},   
    colorlinks=true,  
   linkcolor=purple,         
    citecolor=blue,        
    filecolor=green,      
    urlcolor=blue}          

\usepackage{geometry}
 \usepackage{fullpage}
\usepackage{geometry}
\usepackage{cleveref}
\usepackage{indentfirst}
\newtheorem{theorem}{Theorem}
\newtheorem{lemma}[theorem]{Lemma}
\newtheorem{corollary}[theorem]{Corollary}
\newtheorem{claim}[theorem]{Claim}

\newtheorem{bigthm}{Theorem}

\theoremstyle{definition}
\newtheorem{definition}[theorem]{Definition}
\newtheorem{example}[theorem]{Example}

\newtheorem{remark}[theorem]{Remark}

\numberwithin{theorem}{section}
\numberwithin{equation}{section}
 
\DeclareMathOperator \PSH {{\rm PSH}}

\DeclareMathOperator \exph {{\rm exph}}

\DeclareMathOperator \ric {{\rm Ric}}
\DeclareMathOperator \Null {{\rm Null}}
\DeclareMathOperator \NK {E_\textrm{{nH}}}
\DeclareMathOperator \Z {\textrm{{Sing}}}
\DeclareMathOperator \nh {E_\textrm{{nH}}}

\DeclareMathOperator \tr {tr}

\def\f{\varphi}
\def\dc{\mathrm{dd^c}}

\def\ddbar{\partial\Bar{\partial}}

\begin{document}
	\title[Hermitian null loci]{Hermitian null loci} 
	\author{Quang-Tuan Dang}
\address{Yau Mathmatical Sciences Center, Tsinghua University, Beiing, China 100084} \email{dangquangtuan10@gmail.com $\&$ dangqt@mail.tsinghua.edu.cn}

%   \address{Courant Institute of Mathematical Sciences, New York University
% 251 Mercer St, New York, NY 10012}
% \email{tosatti@cims.nyu.edu}
 
	% %\thanks{The work is partially supported by the ANR project PARAPLUI}
	 	\date{\today}
\begin{abstract}
   % We show that the non-K\"ahler locus of a nef and big (not closed) $(1,1)$-form on a compact complex manifold equals the union of all subvarieties where the restriction of form is not big (null locus). 
   We establish a transcendental generalization of Nakamaye's theorem to compact complex manifolds when the form is not assumed to be closed. 
   We apply the recent 
   analytic technique developed by Collins--Tosatti to
   show that the non-Hermitian locus of a nef and big $(1,1)$-form, which is not necessarily closed, on a compact complex manifold equals the union of all positive-dimensional analytic subvarieties where the restriction of the form is not big (null locus). As an application, we can give an alternative proof of the Nakai--Moishezon criterion of Buchdahl and Lamari for complex surfaces and generalize this result in higher dimensions.
   % This is also used for studying 
   % degenerate complex Monge--Amp\`ere equations on compact Hermitian manifolds. 
   Finally, we investigate finite time non-collapsing singularities of the Chern--Ricci flow, partially answering a question raised by Tosatti--Weinkove.  
   %non-collapsing degenerations of Ricci-flat Hermitian metrics on non-K\"ahler Calabi-Yau varieties.
\end{abstract}

\maketitle

      \tableofcontents

\section{Introduction}

Characterizing the augmented base locus of a nef and big line bundle on a smooth algebraic variety over $\mathbb{C}$ was initially motivated by applications to the diophantine approximation problem~\cite{Nakamaye-book}. Nakamaye showed in~\cite{nakamaye00-base-loci} that the augmented base locus of a nef and big line bundle on a projective manifold equal to its null locus (which is the union of all irreducible subvarieties where the restriction of the bundle is not big). There are numerous generalizations of the latter result in  $\mathbb{R}$-divisors~\cite{Ein-Laz-Mus-Nak-Pop09-base}, in possibly singular varieties~\cite{Birkar17-baselocus,Cacciola-Lopez14-logcanonical,boucksom-cacciola-lopez14-baseloci}, and in positive characteristics~\cite{Casini-McKernan-Mutaca14-baselocus}. We refer interested readers to~\cite{tosatti18-nakamaye} for a comprehensive overview.

Our primary goal in this paper is to extend Nakamaye's theorem to the context of compact non-K\"ahler manifolds with non-closed differential forms.
Let $X$ be a compact complex manifold of dimension $n$, equipped with a Hermitian metric $\omega_X$. 
We denote by ${\rm d^c}=\frac{1}{2i}(\partial-\Bar{\partial})$ so that $\dc=i\ddbar$.
Let $\theta$ be a smooth real (1,1) form and let $[\theta]_{\ddbar}$ denote the  $\partial\Bar{\partial}$-class of forms (resp. currents): $[\theta]_{\ddbar}=\{\theta+\dc \f: u\in \mathcal{C}^\infty(X,\mathbb{R})\; (\text{resp.}\, u\in L^1(X))\}.$ For simplicity, we omit the subscript $\ddbar$.
When $\theta$ is closed, the $\ddbar$-class $[\theta]$ coincides with its Bott--Chern cohomology class in $H^{1,1}_{\rm BC}(X,\mathbb{R})$.
A $\partial\Bar{\partial}$-class $[\theta]$ is said to be {\em Hermitian} (or {\em positive}) if it contains at least a positive definite form, i.e., $\exists\,\f\in\mathcal{C}^\infty(X,\mathbb{R})$ such that $\theta+\dc\f$ is a Hermitian metric.  Although $\theta$ is not closed in general, one can still meaningfully define the concept (in the analytic sense) of nef $\partial\Bar{\partial}$-classes; cf. Def.~\ref{def: nef}. Similarly, a $\partial\Bar{\partial}$-class $[\theta]$ is said to be {\em big} if it contains a positive current that dominates a Hermitian metric. 
It should be emphasized that the previous concepts can be defined when the underlying variety is possibly singular; cf. Sect.~\ref{sect: singularspace}. 
By Demailly's regularization theorem, the Hermitian current in a big class can be taken to have (almost) analytic singularities (in the sense of ~\cite[Theorem 3.2]{demailly2004numerical}). In particular, the Hermitian current is smooth in the complement of a Zariski closed subset of $X$. The smallest possible Zariski closed subset of $X$ that can be taken in this way is the non-Hermitian locus of a big $\ddbar$-class $[\theta]$, denoted by $\NK(\theta)$; cf.~\cite[Definition 3.16]{boucksom2004divisorial}.

A nef and big $\partial\Bar{\partial}$-class $[\theta]$, in general, may not be positive, and another subset to measure its non-positivity is its null locus, denoted by $\Null(\theta)$, explicitly introduced by Nakamaye. The latter subset is defined as the union of all positive-dimensional irreducible analytic subvarieties $V\subset X$ where the restriction $\theta|_V$ is not big. Clearly, the null locus is a proper subset of  $X$. Repeating line by line the proof of Collins--Tosatti~\cite{collins2015kahler}, one can prove that given a nef and big $\ddbar$-class $[\theta]$, its non-Hermitian locus and its null locus coincide. The proof does not require the closedness of the representative form.

\begin{theorem}[{\cite[Theorem 1.1]{collins2015kahler}}]
    Let $X$ be a compact complex manifold. 
    % Assume that $v_-(X,\omega_X)<+\infty$ and $v_-
    % (V,\omega_X|_V)>0$ for any positive dimension submanifold $V\subset X$. 
    Let $\theta$ be a smooth real (1,1) form whose $\ddbar$-class $[\theta]$ is nef and big. Then
    \[\NK(\theta)=\Null(\theta). \]
\end{theorem}

In the K\"ahler case (or the Fujiki class $\mathcal{C}$ in general), the bigness of a closed nef form $\theta$ is equivalent to the positivity property of its self-intersection number, i.e., $\int_X\theta^n>0$.  Unlike the K\"ahler case,
the difficulty we face is that the volume of two forms in the same $\ddbar$-class could be different in general. Let $X$ be a compact complex manifold of dimension $n$, equipped with a Hermitian metric $\omega_X$.
A question raised by Guedj--Lu~\cite{guedj2022quasi2} is whether
\[v_+(X,\omega_X)=\sup\left\{\int_X(\omega_X+\dc\f)^n: \f\in\mathcal{C}^\infty(X), \omega_X+\dc\f>0 \right\}<+\infty; \]
% is finite and 
\[v_-(X,\omega_X)=\inf\left\{\int_X(\omega_X+\dc\f)^n: \f\in\mathcal{C}^\infty(X), \omega_X+\dc\f>0 \right\}>0. \]
% is bounded away from zero. 
We refer to~\cite{angella2022plurisigned} for some classes of manifolds that satisfy the condition above. 
We say that a Hermitian metric $\omega_X$ satisfies the volume condition~\eqref{H} if: \begin{equation*} \tag{$H_X$}\label{H}
    v_+(X,\omega_X)<+\infty \quad\text{and}\; v_-
     (V,\omega_X|_V)>0,
\end{equation*} for all closed positive-dimensional submanifolds $V\subset X$.

% \medskip
% \noindent{\bf Condition H.} \label[H]{H} $v_+(X,\omega_X)<+\infty$ and $v_-
%      (V,\omega_X|_V)>0$ for all positive-dimensional closed submanifold $V\subset X$.

\begin{remark}

It has been shown in~\cite[Theorem A]{angella2022plurisigned} that for any closed  submanifold $V\subset X$, $$v_+(X,\omega_X)<+\infty\Longrightarrow v_+(V,\omega_X|_V)<+\infty.$$
   It is important to emphasize that the assumption~\eqref{H} occurs in some contexts of complex differential geometry. For instance, if $\omega_X$ satisfies the following condition $\dc\omega_X=0$, ${\rm d}\omega_X\wedge {\rm d^c}\omega_X=0$ then by~\cite{chiose2016invariance}, it is equivalent to the volume with respect to a form belonging to the $\ddbar$-class $[\omega_X]$ is unchanged. That is,  for any closed submanifold $V\subset X$,
   \[ \int_V(\omega_X|_V+\dc\f)^{\dim V}=\int_V\omega_X|_V^{\dim V},\forall\;{\omega_X}|_V+\dc\f>0. \]
       These metrics, introduced by Guan and Li~\cite{guan2010complex}, are referred to as the Guan--Li metrics in this article (cf. also~\cite{FT11-astheno}).
   %These metrics were introduced by Guan--Li~\cite{guan2010complex}, which are called the Guan--Li metrics in this paper, although they were probably studied earlier; cf.~\cite{FT11-astheno}.
\end{remark}
Given $\theta$ a smooth real (1,1) form on $X$, if $\theta$ is nef (cf. Def.~\ref{def: nef}), we can define
\[\widehat v_-(X,\theta):=\inf_{\varepsilon>0}v_-(X,(1-\varepsilon)\theta+\varepsilon\omega_X).\]
Although the form $(1-\varepsilon)\theta+\varepsilon\omega_X$ is not positive, one could find, by definition, for any $\varepsilon>0$, a smooth function $\f_\varepsilon$ such that $\theta+\dc\f_\varepsilon\geq -\varepsilon\omega_X$ so that $$(1-\varepsilon)\theta+\varepsilon\omega_X+\dc(1-\varepsilon)\f_\varepsilon\geq\varepsilon^2\omega_X.$$ Hence, the quantity $v_-(X,(1-\varepsilon)\theta+\varepsilon\omega_X)$ can be understood as $v_-(X,(1-\varepsilon)\theta+\varepsilon\omega_X+\dc(1-\varepsilon)\f_\varepsilon)$. Our main theorem is the following.
% A characterization of a nef and big class is the following
\begin{bigthm}\label{thmA}
    Let $X$ be a compact complex manifold equipped with a Hermitian metric $\omega_X$. 
    % Let $\theta$ be a smooth real (1,1) form whose $\ddbar$-class $[\theta]$ is nef and big.
     Assume that $\omega_X$ satisfies the condition~\eqref{H}. Let $\theta$ be a 
     smooth real (1,1) form whose $\ddbar$ class is nef with $\widehat v_-(X,\theta)>0$. 
     Then $[\theta]$ is big and
     \begin{equation}\label{eq: NH=Null=}
         \NK(\theta) =\Null(\theta)=\bigcup_{\widehat v_-(V,\theta|_V)=0} V,
     \end{equation}
     where the union is taken over all purely positive-dimensional irreducible analytic subvarieties $V$ of $ X$.
\end{bigthm} 
In particular, when $\theta$ is closed or merely satisfies the so-called Guan--Li condition: $\dc\theta=0$, ${\rm d}\theta \wedge {\rm d^c}\theta=0$ (equivalent to $\dc\theta^k=0$ for all $k=1,\dots,n-1$), then we obtain $\widehat v_-(V,\theta|_V)=\int_V\theta^{\dim V}$ by~\cite[Lemma 4.4]{guedj2022quasi2}, provided that $v_+(X,\omega_X)<\infty$. This directly implies the Collins--Tosatti result in the case of compact complex manifolds belonging to the Fujiki class $\mathcal{C}$.

\medskip 
A direct application of Theorem~\ref{thmA} is to establish the Nakai--Moishezon criterion for some compact non-K\"ahler manifolds, which generalizes the
fundamental result of Demailly and P\u aun ~\cite{demailly2004numerical} for compact K\"ahler manifolds, and Buchdahl~\cite{buchdahl00-NakaiMoishezon} and Lamari~\cite{lamari1999courants} for compact non-K\"ahler surfaces. Our proof relies essentially on the mass concentration result of Demailly--P\u aun, whose proof was simplified by Chiose~\cite{chiose2016kahler} and Popovici~\cite{popovici2016sufficients}. Using Theorem~\ref{thmA} allows one to avoid the inductive argument on dimension established in~\cite{demailly2004numerical}.  
% There are some major difficulties in adapting the method of Buchdahl and Lamari to higher dimensions. 
Inspired by the works of Collins--Tosatti~\cite{collins-tosatti-singular} and Das--Hacon--P\u aun~\cite{das-hacon-paun2022mmp}, we establish the Nakai--Moishezon criterion on a possibly singular compact analytic normal variety, but the proof relies on the induction on dimension. 
\begin{bigthm}\label{thm: N-M criterion}
     Let $X$ be a compact analytic normal variety of pure dimension $n$ and $\omega_X$ a Hermitian metric on $X$ such that $\dc\omega_X^k=0$ for all $k=1,\ldots,n-1$. Let $\theta$ be a smooth real (1,1) form on $X$ such that
    \begin{enumerate}
    \item $\dc\theta^k=0$, for any $k=1,\dots, n-1$;
        \item for all $k=1,\ldots, n$, $$\int_X\theta^k\wedge\omega_X^{n-k}>0;$$
        \item  for every irreducible analytic
        subvariety $V\subset X$ with $\dim V>0$, 
        $$\int_V\theta^k\wedge \omega_X^{\dim V-k}>0,\,\;\forall\,k=1,\dots,\dim V .$$
    \end{enumerate} Then the $\ddbar$-class $[\theta]$ is Hermitian.
\end{bigthm}
 We note that condition (2) can be absorbed in (3) when $V=X$. In the case of complex smooth surfaces, thanks to~\cite[Proposition 5]{Buchdahl1999compact}, we can relax the condition $(3)$ to require only that $\int_V\theta>0$ for every irreducible curve $V\subset X$ with strictly negative self-intersection $V\cdot V<0$. 
 We refer to Sect.~\ref{sect: singularspace} for the definition of differential forms, Hermitian metrics, and integrals on singular analytic varieties.
% Inspired by the work of Collins and Tosatti~\cite{collins-tosatti-singular}, we expect to prove a generalized version of Theorem~\ref{thm: N-M criterion} for possibly singular compact complex varieties embedded in a smooth ambient space. We believe this holds and leave it to forthcoming work.
% Here $NK(\theta)$ denotes the non-K\"ahler locus of $\theta$ and $Null(\theta)$ the null locus. We can show by assumption that $NK(\theta)$ and $Null(\theta)$ are proper subsets of $X$. 

\medskip
Finally, we study finite time singularities of the Chern--Ricci flow on compact complex manifolds, partially answering a question raised by Tosatti--Weinkove. 
% Let $(X,\omega_0)$ be a compact complex manifold of dimension $n$.
Let $\omega=\omega(t)$ be a solution of the Chern--Ricci flow on $X$, for $t\in[0,T[$, starting with a Hermitian metric $\omega_0$,
\begin{equation}
    \label{crf0}
    \frac{\partial\omega}{\partial t}=-\ric(\omega), \quad\omega(0)=\omega_0,
\end{equation}
where $\ric(\omega)$ is the Chern--Ricci form (also called the first Chern form) of the Hermitian metric $\omega$.
It is classical that solving the Chern--Ricci flow boils down to solving a nonlinear parabolic scalar equation in $\f=\f(t)$ of the form
\begin{equation*}
    \frac{\partial\f}{\partial t}=\log\frac{(\theta_t+\dc\f)^n}{\omega_0^n},\quad\theta_t+\dc \f>0,\;\f(0)=0,
\end{equation*}
where $\theta_t:=\omega_0- t\ric(\omega_0)$.
Tosatti and Weinkove~\cite{tosatti2015evolution} characterized the maximal existence time $T$ of the flow as
\begin{equation*}
    T=\sup\{t>0:\exists\,\psi\in\mathcal{C}^\infty(X)\;\text{with}\;\theta_t+\dc\psi>0\},
\end{equation*}
by this, we mean that $T$ is such that for all $t<T$, the $\ddbar$-class $[\theta_t]$ is Hermitian.
Suppose that the flow develops a finite time singularity $T<\infty$. We observe that the limiting  $\ddbar$
class $[\theta_T]$ is nef but not Hermitian.%(the flow cannot exist beyond time $T$)
 We define the set of singularities of the flow, denoted by $\Z(\omega(t))$, to be the complement of the set of all points $x\in X$ such that there is a neighborhood $U\ni x$ in which $\omega(t)$ converges smoothly to a Hermitian metric $\omega_T$ as $t\to T^-$. 
Furthermore, we can show that the singularity set of the Chern--Ricci flow is equal to the set on which the Chern scalar curvature blows up. The latter was studied by Gill--Smith~\cite{gill-smith14-chernricci}, generalizing Zhang's analog result for K\"ahler--Ricci flows.

 Tosatti and Weinkove~\cite[Question 6.1]{tosatti2022chern} posed a question of whether singularities of the Chern--Ricci flow develop precisely along analytic subvarieties of $X$?  
 In the K\"ahler case, this question was initially raised by Feldman, Ilmanen, and Knopf and later answered affirmatively by Collins and Tosatti~\cite[Theorem 1.5]{collins2015kahler}. Gill and Smith~\cite[Theorem 1.2]{gill-smith14-chernricci} gave an affirmative answer
 for the case when $X$ is a complex surface and $\omega_0$ is Gauduchon. As observed in~\cite[Secion 6]{tosatti2022chern}, the answer is affirmative when the volume with respect to $\omega(t)$ shrinks to zero as $t\to T$. We investigate non-collapsing the case when  $\widehat v_-(X,\theta_T)>0$  and provide an answer to this question on some underlying manifolds.
% \begin{bigthm}\label{thm-parabolic}
%     Let $X$ be a compact complex manifold equipped with a Hermitian metric $\omega_X$. 
%      Assume that $\omega_X$ satisfies the assumption~\eqref{H}. Assume that the solution $\omega(t)$ the Chern--Ricci flow~\eqref{crf0} starting at a Hermitian metric $\omega_0$ develops
% a singularity at finite time $T$, and that $\widehat v_-(X,\theta_T)>0$. Then as $t\to T^-$, the singularities of the flow can be characterized as 
% \[\Z(\omega(t))=\bigcup_{\widehat v_-(V,\theta_T|_V)=0} V\]
%      where the union is taken over all purely positive-dimensional irreducible analytic subvarieties $V$ of $ X$.  
% \end{bigthm}
    \begin{bigthm}\label{thm-parabolic}
 Let $(X,\omega_0)$ be a compact Hermitian manifold of dimension $n$. Let $\omega(t)$ be a solution of the Chern--Ricci flow starting at $\omega_0$, which develops
a singularity at finite time $T$. Assume that the limiting form $\theta_T=\omega_0-T\ric(\omega_0)$ satisfies $\widehat v_-(X,\theta_T)>0$.
Then if either 
\begin{itemize}
    \item $n=2$ and $\omega_0$ is arbitrary,
   \item  or $n\geq 3$, $\dc\omega_0=0$, ${\rm d}\omega_0\wedge {\rm d^c}\omega_0=0$, 
\end{itemize}
then the metrics $\omega(t)$ develop singularities precisely along the analytic set $\Null(\theta_T)$.
\end{bigthm}
This theorem generalizes Gill--Smith's result when the underlying manifold $X$ is a surface, but the initial metric $\omega_0$ is given to be arbitrary and also partially answers the question of Tosatti and Weinkove.  
 With assumptions in this theorem, we have that the limiting form $\theta_T$ is big, so its null locus $\Null(\theta_T)$ is a proper analytic subset of $X$ (since $[\theta_T]$ is not Hermitian), and our theorem shows that no singularities develop on its complement. We have seen from Theorem~\ref{thmA} that $\Null(\theta_T)$ is precisely the union of all purely positive-dimensional irreducible subvarieties $V$ of $X$ with $v_-(V,\omega(t))$ shrinking to zero as $t\to T^-$, so the singularities of the flow develop at least there. We will see in Theorem~\ref{thm: sing-formation} that the scalar curvature of $\omega(t)$ is locally uniformly bounded on compact sets outside singularities. Furthermore, it should be emphasized that in the setting of our theorem, as $t\to T^-$, the metrics $\omega(t)$ weakly converge to a positive current $\omega_T$ that is smooth on $X\backslash\Null(\theta_T)$ and has minimal singularities.

 A crucial case of such finite time singularities occurs when the $\ddbar$-class $[\theta_T]$ contains a pullback of a Hermitian metric via the blowdown of an exceptional divisor to a point. Interesting results were obtained by Tosatti and Weinkove~\cite{tosatti2013chern}, T\^o~\cite{to2018regularizing}, and Nie~\cite{nie2017weak}, in which they proved that the Chern--Ricci flow performs a {\em canonical surgical contraction}.

\medskip
This paper is structured as follows. In Sect.~\ref{sect: preliminary},  we introduce notation and definitions and briefly review fundamental results extended to the non-K\"ahler context. In Sect.~\ref{sect: null}, we establish an extension-type theorem for Hermitian currents with analytic singularities, which enables us to prove Theorem~\ref{thmA}, and we also prove Theorem~\ref{thm: N-M criterion} there. 
% Sect.~\ref{sect: elliptic} is dedicated to the study of degenerate complex Monge--Amp\`ere equations, where we prove Theorems~\ref{thm-elliptic} and~\ref{thm-envelope}.
Finally, we investigate finite time singularities of the Chern--Ricci flow and prove Theorem~\ref{thm-parabolic} in Sect.~\ref{sect: crf}.

\subsection*{Notation} Unless otherwise mentioned, $X$ denotes a compact Hermitian manifold with a reference Hermitian metric $\omega_X$. 
We interchangeably use the Hermitian metric $\omega_X$ and its associated Riemannian metric $g$.
We denote by ${\rm d}=\partial+\Bar{\partial}$ and ${\rm d^c}=\frac{i}{2}(\Bar{\partial}-\partial)$ so that
$\dc=i\partial\Bar{\partial}.$
Throughout the paper, we use $C$ to represent a positive constant, which may vary from line to line and can be uniformly controlled.
\subsection*{Acknowledgement} 
 We would like to express our gratitude to V. Tosatti for suggesting this problem, his interest in this work, and his valuable comments. 
We are grateful to T.-D. T\^o and C. H. Lu for insightful discussions. We also thank the anonymous referee for the valuable comments. This work is partially supported by the ANR project PARAPLUI.

\subsection*{Ethics declarations} The author declares no conflict of interest.

\section{Preliminaries}\label{sect: preliminary}
Let $X$ be a compact complex manifold of dimension $n$. Fix $\omega_X$ a reference Hermitian metric on $X$. 
% We let $\mathcal{A}^{p,q}(X,\mathbb{C})$ denote the sheaf of germs of smooth differential $(p,q)$-forms on $X$. 
\subsection{$\partial\Bar{\partial}$-classes}
 It is well-known that the $\partial\Bar{\partial}$-lemma does not hold on an arbitrarily complex manifold, so it is better to work with the $\partial\Bar{\partial}$-class of differential forms.  We denote by $\mathcal{A}^{p,q}(X,\mathbb{C})$ (resp. $\mathcal{A}^{p,q}(X,\mathbb{R})$) the space of smooth (resp. real) forms of bidgree $(p,q)$. We say that two forms $\theta,\eta\in\mathcal{A}^{p,q}(X,\mathbb{C})$ are equivalent if there exists a smooth form $\psi\in\mathcal{A}^{p-1,q-1}(X,\mathbb{C})$ such that $$\theta-\eta=\dc\psi.$$  Since then, we let $[\theta]_{\ddbar}$ denote the $\ddbar$-class with smooth representative $\theta$. For simplicity, we skip the subscript $\ddbar$, and unless we emphasize the $\ddbar$ class $[\theta]$ is interpreted as the set of forms $\theta+\dc\psi$ for all $\psi\in\mathcal{C}^\infty(X)$.

% We remark that a ${\rm d}$-exact form is not $\ddbar$-exact in the general context of non-K\"ahler manifolds. A compact complex manifold $X$ is said to be a {\em $\ddbar$-manifold} if it satisfies the exactness property: for any $\rm d$-closed
% pure-type form $u$ on $X$, 
% $u$ is $d$-exact $\Longleftrightarrow$ $u$ is $\ddbar$-exact for any $d$-closed form $u$ on $X$.
% In particular, a compact K\"ahler manifold is a $\ddbar$-manifold.
% We are only interested in the forms that satisfy the Guan-Li condition: $\dc\theta=\dc\theta^2=0$. In particular, for any real smooth (1,1) form $\theta$ satisfying the Guan--Li condition, we have $\dc\theta^k=0$ for any $1\leq k\leq n$.
% We let $GL^{1,1}(X)$ denote the set of all equivalence classes of Guan--Li form on $X$.
% \[GL^{1,1}(X)=\frac{\{\theta\in\mathcal{C}^{\infty}_{1,1}(X,\mathbb{R}): \dc\theta=\dc\theta^2=0\}}{\dc\mathcal{C}^\infty(X,\mathbb{R})}. \]
% For any real smooth (1,1) form $\theta$ satisfying the Guan--Li condition, we have
% \[ \int_X\theta^n=\int_X(\theta+\dc\psi)^n,\; \forall\, \psi\in\mathcal{C}^\infty(X,\mathbb{C}).\]
% This allows us to give a good definition of the {intersection number} of Guan--Li equivalence classes, denoted by $[\theta]^n=\int_X\theta^n$. 
We refer the interested reader to~\cite[Sect. 1]{BoucksomGuedjLu2025-volume} for more details.
\begin{definition}\label{def: nef}
    A smooth real (1,1) form $\theta$ is called {\em nef} if for any $\varepsilon>0$ 
    there exists a form $\theta_\varepsilon\in[\theta]$ such that $\theta_\varepsilon\geq -\varepsilon\omega_X$, i.e., 
$\exists\,\psi_\varepsilon\in\mathcal{C}^\infty(X,\mathbb{R})$ such that $\theta+\dc\psi_\varepsilon\geq -\varepsilon\omega_X$. In this case, the $\ddbar$-class $[\theta]$ is said to be nef.

If there is a Hermitian metric $\omega\in [\theta]$, then we say that the $\ddbar$-class $[\theta]$ is {\em Hermitian}.
    \end{definition}

 It should be emphasized that, unlike in the K\"ahler case, even if two Hermitian metrics $\omega$ and $\omega'$ lie in the same $\ddbar$-class, i.e., there exists a function $\f \in \mathcal{C}^\infty(X)$ such that $\omega' = \omega + \dc \f$, the volumes with respect to $\omega$ and $\omega'$ are generally not equal. Chiose~\cite[Theorem 0.1]{chiose2016invariance} showed that the preservation of Monge--Amp\`ere masses is equivalent to the condition introduced by Guan–Li~\cite{guan2010complex}: namely, that $\int_X\omega^n=\int_X\omega_X^n$ for all $\omega\in [\omega_X]$ Hermitian if and only if $\dc\omega_X=0$ and ${\rm d}\omega_X\wedge {\rm d^c}\omega_X=0$. The latter condition is equivalent to $\dc\omega_X^k=0$ for all $k=1,\dots, n-1$; see also~\cite[Theorem 0.2]{chiose2016invariance} for additional equivalent conditions.

We recall some quantities on Monge--Amp\`ere volume bounds introduced by Guedj and Lu~\cite{guedj2022quasi2}.

\begin{definition} We define
    \[v_+(X,\omega_X)=\sup\left\{\int_X\omega^n: \forall\,\omega\in[\omega_X],\,\omega\,\text{Hermitian} \right\}, \]
 and 
\[v_-(X,\omega_X)=\inf\left\{\int_X\omega^n: \forall\,\omega\in[\omega_X],\,\omega\,\text{Hermitian} \right\}. \]
\end{definition}
A question raised by Guedj and Lu~\cite{guedj2022quasi2} is whether $v_+(X,\omega_X)<+\infty$ and/or $v_-(X,\omega_X)>0$. We remark that these conditions somewhat depend on the complex structure, but they are independent of the choice of a Hermitian metric; see \cite[Proposition 3.2]{guedj2022quasi2}. Furthermore, they are bimeromorphic invariants, and so is the condition~\eqref{H}.
\begin{theorem}[{\cite[Theorem 3.7]{guedj2022quasi2}}]
\label{thm: invariant-H} Let $\pi\colon X \to Y$ be a bimeromorphic map between compact complex manifolds equipped with Hermitian metrics $\omega_X$ and $\omega_Y$. Then $\omega_X$ satisfies the condition~\eqref{H} if and only if $\omega_Y$ satisfies the condition~\eqref{H}.
\end{theorem}
 % We refer to~\cite{angella2022plurisigned} for some classes of manifolds satisfying the condition above. 
% If $\theta$ is a Hermitian metric, the supremum and infimum in the definition above can be taken over all smooth $\theta$-psh functions, thanks to Demailly's regularization theorem and Bedford--Taylor's convergence theorem.
% \begin{hypo} \label{H}
%     We say that $\omega_X$ satisfies the volume condition $(H)$ if: \begin{equation*} \tag{H}\label{H}
%     v_+(X,\omega_X)<+\infty \quad\text{and}\; v_-
%      (V,\omega_X|_V)>0
% \end{equation*} for any positive-dimensional compact submanifold $V\subset X$.
% \end{hypo}
\begin{definition}
    If $\theta$ is nef, we can define
\[\widehat v_-(X,\theta)\coloneqq\inf_{\varepsilon>0}v_-(X,(1-\varepsilon)\theta+\varepsilon\omega_X).\]
\end{definition} 
Although the form $(1-\varepsilon)\theta+\varepsilon\omega_X$ is not positive, one could find by definition, for any $\varepsilon>0$, a smooth function $\f_\varepsilon$ such that $\theta+\dc\f_\varepsilon\geq -\varepsilon\omega_X$ so that $$(1-\varepsilon)\theta+\varepsilon\omega_X+\dc(1-\varepsilon)\f_\varepsilon\geq\varepsilon^2\omega_X.$$ Hence the quantity $v_-(X,(1-\varepsilon)\theta+\varepsilon\omega_X)$ is interpreted as $v_-(X,(1-\varepsilon)\theta+\varepsilon\omega_X+\dc(1-\varepsilon)\f_\varepsilon)$. We expect that the above definition coincides with the one when $\theta $ is closed. The necessary condition is that $\theta$ satisfies the Guan--Li property.
\begin{lemma}[{\cite[Lemma 4.4]{guedj2022quasi2}}]\label{lem: GL4.4}
    If $v_+(X,\omega_X)<\infty$ and $\theta$ is a smooth nef (1,1) form satisfying $\dc\theta^k=0$ for all $k=1,\ldots,n-1$ then $\widehat v_-(X,\theta)=\int_X\theta^n$. Moreover, for any closed positive-dimensional submanifold $V\subset X$, $\widehat v_-(V,\theta|_V)=\int_V\theta^{\dim V}$.
\end{lemma}
The proof is identical to that of~\cite[Lemma 4.4]{guedj2022quasi2}, which we omit here. Notice that the Guan--Li property is sufficient to guarantee the preservation of Monge--Amp\`ere masses.

\subsection{Quasi-plurisubharmonic functions and Hermitian currents}
 Recall that an upper semi-continuous function $\f:X \rightarrow \mathbb{R}\cup\{-\infty\}$ is {\em quasi-plurisubharmonic} ({\em quasi-psh} for short) if it
is locally described as the sum of a plurisubharmonic (psh) function and a smooth function. 
% In other words, for any $x\in X$, there are a neighborhood $U$ of $x$, a psh function $\rho$ on $U$, and a smooth function $f$ on $U$ such that $\f|_U=\rho+f$.
\begin{definition}
    Let $\theta$ be a smooth (1,1) form on $X$. One says that 
    \begin{itemize}
        \item a quasi-psh function $\f$ is  {\em $\theta$-plurisubharmonic} ({\em $\theta$-psh} for short) if $\theta+\dc \f\geq 0$ in the sense of currents. We let $\PSH(X,\theta)$ denote the space of all $\theta$-psh functions on $X$ which are not identically $-\infty$;
        \item  $T$ is a positive (1,1) current on $X$ in the $\ddbar$-class $[\theta]$ if $T=\theta+\dc\f$ for some $\f\in\PSH(X,\theta)$, and vice versa.
    \end{itemize}
    For any positive (1,1) current $T$, we let $\Z(T)$ denote the {\em singular set} of $T$, namely, if $T=\theta+\dc\f$ then $\Z(T)=\{\f=-\infty\}$.
\end{definition}
 We refer the reader to nice references~\cite{demaillycomplex,guedj2017degenerate} for more details about the basic properties of psh functions and positive currents. 

\medskip
  The {\em complex Monge--Amp\`ere measure} $(\theta+\dc u)^n$ is well defined for any 
$\theta$-psh function $u$, which is {\em bounded} on $X$, as follows from Bedford--Taylor's theory; cf.~\cite{dinew2012pluripotential} or~\cite{guedj2022quasi2} for a brief recap. We recall the following maximum principle.

\begin{lemma}[Maximum principle] \label{lem: max-princ}
    Let $\f$, $\psi$ be bounded $\theta$-psh functions in $U\subseteq X $ such that $\f\geq \psi$. Then
    \[ \mathbf{1}_{\{\f=\psi\}}(\theta+\dc\f)^n\geq\mathbf{1}_{\{\f=\psi\}}(\theta+\dc\psi)^n.\]
\end{lemma}
\begin{proof}
    The proof directly follows Bedford--Taylor's maximum principle~\cite[Theorem 3.27]{guedj2017degenerate} in the local context. See also~\cite[Lemma 1.2]{guedj2022quasi2} for a proof in the global one.
\end{proof}
\begin{definition} Let $\mathcal{I}\subset\mathcal{O}_X$ be an analytic coherent ideal sheaf and $c>0$.
    A quasi-psh function $\f$ on X is said to have {\em almost analytic singularities}  of type $(\mathcal{I},c)$ if the following are satisfied 
    \begin{enumerate}[label=(\roman*)]
          \item for any $x\in X$, there exists a neighborhood $U$ of $x$  such that on $U$ we have
        \[ \f=c\log \left(\sum_j|f_j|^2 \right)+g,\]
        where the $f_j$'s are local generators of a coherent sheaf $\mathcal{I}$, and $g$ is a bounded function on $U$, 
        % is smooth outside the analytic subset $V(\mathcal{I})$ induced by $\mathcal{I}$,
        \item there is a (global) proper modification $\pi: X'\to X$ of $X$, obtained as a finite composition of blow-ups with smooth
centers, and an effective divisor $\sum_\ell a_\ell D_\ell $ with normal crossings
       such that one can write $\f\circ\pi$ locally as
       \[ \f\circ \pi=c\sum_\ell a_\ell \log |g_\ell|^2 +h,\]
        where $(g_\ell=0)$ are local equations of the divisors $D_\ell$ and  
       $h$ is a smooth function.
    \end{enumerate}      
       % \item there is a (global) proper modification $\pi: X'\to X$ of $X$ such that $\pi^*\mathcal{I}$ is the principal ideal, generated by a  normal crossing divisor $\sum_j a_jD_j$ with $a_j> 0$, and for any $x'\in X'$, there exists a local chart $W$ and holomorphic coordinates $z'=(z_1',\ldots, z_n')$ around $x'$, with $ \cup D_j\cap W=\{z'_1\cdots z'_p=0\}$, $1\leq p\leq n$, we have\[\f\circ \pi(z')=c\left(\sum_{j=1}^p a_j\log|z_j'|^2\right)+h. \]
If a current $T=\theta+\dc\f$ with $\f$ a quasi-psh function having almost analytic singularities (along subscheme $V(\mathcal{I})$ induced by $\mathcal{I}$), then one says that $T$ has {\em almost analytic singularities}.
    % In this case, we denote by $Z(T)$ the singular locus of $T$. 
\end{definition}
% The condition $g$ is smooth outside $V(\mathcal{I})$ is superfluous since it indeed follows from $(ii)$.
We emphasize that the definition above is a bit lengthy, but turns out to be the most
suitable one for this paper. We will see below that a $\ddbar$-class with appropriate positivity contains plenty of currents with such singularities.
\begin{definition}
    We say that a real (1,1)-current $T$ is a {\em Hermitian current} if there exists a Hermitian metric $\omega$ such that $T\geq \omega$. 

    The $\ddbar$-class $[\theta]$ is said to be {\em big} if it contains a Hermitian current, i.e., there exists $\f\in\PSH(X,\theta)$ such that $T=\theta+\dc\f\geq \omega$ for some Hermitian form $\omega$. In this case, the smooth representative $\theta$ is also called big. 
\end{definition}
 We introduce the notion of psh envelopes understood as the {\em least singularity} function lying below a prescribed one.
\begin{definition}\label{def: envelope}
    Given a Borel function $f:X\to\mathbb{R}$, we define the {\em $\theta$-psh envelope} of $f$ by
    \[ P_\theta(f):=(\sup\{u\in\PSH(X,\theta): u\leq f \})^*,\]
    where the $^*$ means that we take the upper semi-continuous regularization.
\end{definition}
The $\theta$-psh envelope $P_\theta(f)$ might be $\equiv -\infty$ when the $\ddbar$-class $[\theta]$ contains no positive currents. When $\theta$ is big, and $f$ is bounded from below, $P_\theta(f)$ is well defined as a $\theta$-psh function. 

% \medskip
We recall the basic type of regularization due to Demailly; cf.~\cite[Lemma 2.1, Theorem 3.2]{demailly2004numerical}.
\begin{theorem}[Demailly's regularization theorem] \label{thm: regularization} Let $X$ be a compact complex manifold equipped with a Hermitian form $\omega$.
    Let $T=\theta+\dc\f $ be a real (1,1) current; i.e., $\theta$ is a smooth real (1,1) form and $\f$ is quasi-psh.  
    Suppose that $T\geq \gamma$ for some smooth (1,1) form $\gamma$. Then, there exists a decreasing sequence of quasi-psh functions $\f_m$ such that \begin{enumerate}
        \item $T_m=\theta+\dc\f_m$ are currents with {almost analytic singularities};
        \item $T_m\to T$ weakly, and $\f_m$ decreases to $\f$ as $m\to\infty$;
        \item $T_m\geq \gamma-\varepsilon_m\omega$ for some sequence $\varepsilon_m\searrow 0$.
    \end{enumerate}
\end{theorem}
By replacing $T$ (resp. $\gamma$) with $T-\theta$ (resp. $\gamma-\theta$) we can work with $T=\dc\f$ a closed quasi-positive current.
The proof relies on the appropriate technique using an asymptotic Bergman kernel procedure as we describe what happens locally. If $\f$ is a psh function on the ball $B\subset \mathbb{C}^n$ one considers for any $m\geq 1$ the Hilbert space $\mathcal{H}(B,m\f)$ of $L^2$ holomorphic function on $B$ such that $\int_B|f|^2e^{-m\f}dV<\infty$. We choose a Hilbert basic $(f_k^m)_{k\in\mathbb{N}}$ of $\mathcal{H}(B,m\f)$ and put $\f_m:=\frac{1}{m}\log\left(\sum_{k\in \mathbb{N}}|f_k^m|^2 \right)$. One can check that $\f_m$ is well defined, psh, with analytic singularities, and almost increasing. The difficult part then is to prove that $\f_m$ converges to $\f$ as $m$ goes to infinity. To get the global approximations, one glues the local ones by using a partition of unity argument. In this gluing process, a small loss of positivity is established as a result of the change of holomorphic coordinates. This is where the $\varepsilon_m$'s come from. For each $m$, the corresponding modification $\mu_m$ is then obtained by principalization of singularities of the multiplier ideal sheaf $\mathcal{I}(m\f)$ so that $\mu_m^*\mathcal{I}(m\f)$ is the principal ideal sheaf associated with a normal crossing divisor, due to the fundamental result of Hironaka; cf.~\cite{wlodarczyk09-Resolution}. 

\medskip
From the above regularization, 
%we are able to approximate a positive current by Hermitian currents with almost analytic singularities, with reasonable control of singularities.  In particular, 
a big $\ddbar$-class $[\theta]$ contains various Hermitian currents with almost analytic singularities.
% : if $T\in [\theta]$ such that $T\geq \omega$ then $T_m\geq (1-\varepsilon_m)\omega$. 

We introduce the definition of the non-Hermitian locus of a big $\ddbar$-class, analogous to the one in~\cite{boucksom2004divisorial}. 
\begin{definition}[non-Hermitian locus] If $\theta$ is a big (1,1) form on $X$, we define its {\em non-Hermitian locus} to be the set $$\NK(\theta)=\bigcap_{T\in [\theta]}\Z(T),$$
    where the intersection is taken over all Hermitian currents with almost analytic singularities in the $\ddbar$-class $[\theta]$.
\end{definition}

We investigate the bigness in some particular context of a fixed Hermitian metric.
\begin{theorem}[Guedj--Lu] \label{thm: guedjlu}
     Let $X$ be a compact complex manifold equipped with a Hermitian metric $\omega_X$
     satisfying $v_+(\omega_X)<+\infty$. Let $\theta$ be a nef (1,1) form with $\widehat v_-(X,\theta)>0$. Then the $\ddbar$-class $[\theta]$ is big, i.e., there exists a Hermitian current $T$ in $[\theta]$ with almost analytic singularities.
\end{theorem}
\begin{proof}
    The proof was given in~\cite[Theorem 4.6]{guedj2022quasi2}, which was inspired by Chiose \cite{chiose2016kahler} and Popovici \cite{popovici2016sufficients}
    based on the characterization result of Lamari \cite{lamari1999courants}. 

We note that, without the nefness condition, Boucksom, Guedj and Lu~\cite{BoucksomGuedjLu2025-volume} recently proved that the positivity of the lower volume of $\theta$ implies the existence of Hermitian current. 
\end{proof}
\begin{theorem}\label{thm: boucksom317}
    Let $\theta$ be a big (1,1)-form. Then there exists a Hermitian current $T_0$ with almost analytic singularities such that $\NK(\theta)=\Z(T_0)$. 
    In particular, the non-Hermitian locus of a big class $[\theta]$ is an analytic subvariety of $X$. 
\end{theorem}

\begin{proof} The proof is quite close to that of~\cite[Theorem 3.17]{boucksom2004divisorial}.
We denote by $\mathcal{S}$ the set of all Hermitian currents with almost analytic singularities in the class $[\theta]$. Fix $T_j=\theta+\dc\f_j\in\mathcal{S}$ for $j=1,2$. First, we claim that there exists a Hermitian current with almost analytic singularities $T$ such that $\Z(T)\subset \Z(T_1)\cap \Z(T_2)$. Indeed, if we set $\f_3=\max(\f_1,\f_2)$ then $T_3=\theta+\dc\phi_3$ is a Hermitian current such that $(\f_3=-\infty)\subset \Z(T_1)\cap \Z(T_2)$. By Theorem~\ref{thm: regularization}, we can find a Hermitian current $T=\theta+\dc\f\in \mathcal{S}$  with almost analytic singularities such that $\f\geq \f_3$ so $\Z(T)\subset \Z(T_1)\cap \Z(T_2)$. This proves our claim.

Thanks to the claim above and Theorem~\ref{thm: regularization} again, we can construct a sequence of currents $T_m=\theta+\dc\f_m\in \mathcal{S}$  such that $\Z(T_m)$ is a decreasing sequence of analytic subsets. By the strong Noetherian property, $\Z(T_m)$ must be stationary for $m\geq m_0$. Therefore, we obtain $\NK(\theta)=\Z(T_{0})$ for $T_0\coloneqq T_{m_0}$, as desired. 

Moreover, if $\NK(\theta)=\varnothing$ then $\Z(T_0)=\varnothing$, which means that $T_0$ is smooth on $X$, so $T_0$ is in fact a Hermitian form.
\end{proof}
In summary, we have the following.
\begin{theorem}\label{thm: current-anasing}
    Let $X$ be a compact complex manifold equipped with a positive form $\omega_X$ satisfying $v_+(\omega_X)<+\infty$. Let $\theta$ be a nef (1,1) form with $\widehat v_-(X,\theta)>0$. Then there exists a Hermitian current $T$ in the $\ddbar$-class $[\theta]$ with almost analytic singularities such that
    \[\Z(T)=\NK(\theta).\]
\end{theorem}
We recall the following lemma, which implies that the non-Hermitian locus of a big form has no isolated points.
\begin{lemma}~\label{lem: isolated}
     If $\theta$ is a big (1,1)-form, then the analytic set $\nh(\theta)$ has no isolated points.
\end{lemma}
\begin{proof}
   We refer to~\cite[Lemma 3.1]{collins2015kahler}. The closedness is not involved.
\end{proof}
\subsection{Extension on complex analytic spaces}\label{sect: singularspace}
We give here a brief recap of the extension of pluripotential theory to compact complex spaces. We refer the interested readers to~\cite{demailly1985measures} for more details.

Let $X$ be a reduced complex analytic space of pure dimension $n$.  We will let $X_{\rm reg}$ denote the locus of regular points of $X$, and $X_{\rm sing}:=X\backslash X_{\rm reg}$ the locus of singular points. We see that $X_{\rm reg}$ is a complex manifold of dimension $n$ and $X_{\rm sing}$ is an analytic subset of $X$ of complex codimension $\geq$ 1. If $X=\bigcup_\alpha U_\alpha$ is an open covering,
we consider local embeddings $j_\alpha:U_\alpha\hookrightarrow\mathbb{C}^N$ into analytic subsets of $\mathbb{C}^N$ for some $N\geq 1$. By these local embeddings, one can define the sheaves of smooth $(p,q)$-forms $\mathcal{A}^{p,q}$, and the space of currents is, by definition, the dual of the space of differential forms as in the smooth case. The differential operators ${\rm d}$, $\partial$, $\Bar{\partial}$, ${\rm d^c}$, and $\dc$ satisfy the usual rules and are then well defined by duality.
We can also define the concepts of Hermitian $\ddbar$-classes, nef $\ddbar$-classes, big $\ddbar$-classes, etc.
% We next move on to psh functions. There are essentially two analytic notions which extend the usual one for complex manifolds. 

 We next move on to the notion of psh functions. Similarly, by the local embeddings above, we say that a function $u:X\to\mathbb{R}\cup\{-\infty\}$ is psh on $X$ if it is locally the restriction of a psh function defined an open set $\Omega_\alpha$ of $\mathbb{C}^N$, which contains $j_\alpha(U_\alpha)$.

 Another way to define a psh function is introduced by Demailly: a function $u$ is weakly psh on $X$ if it is locally bounded from above on $X$ and its restriction to the complex manifold $X_{\rm reg}$ is psh. We can extend it to $X$ by $$u^*(x):=\limsup_{X_{\rm reg}\ni y\to x}u(y).$$ The function $u^*$ is upper semi-continuous, locally integrable on $X$, and $\dc u^*$ is a positive current. Since $X$ is locally irreducible, thanks to ~\cite[Th\'eor\`eme 1.10]{demailly1985measures}, $u^*$ is psh on $X$.   
 In the same way, one can define the notion of quasi-psh functions with almost analytic singularities, Hermitian currents, non-Hermitian loci, etc. 

 We end this section by discussing the integration over analytic (irreducible) subvarieties of a smooth complex manifold $X$. If $V$ is an analytic subvariety of $X$ of pure dimension $k>0$, for any smooth $(k,k)$ form $\alpha$ on $X$, then the integral
 \[\int_V\alpha:=\int_{V_{\rm reg}}\alpha \]
  is finite thanks to the fundamental result of Lelong; see, e.g.,~\cite[page 32]{griffiths1978principles}. In the same way, we can define the quantities $v_-(V,\omega_V)$, $v_+(V,\omega_V)$ and $\widehat v_-(V,\theta)$.
  Furthermore, for a smooth real $(k-1,k-1)$-form $\beta$, we have by~\cite[page 33]{griffiths1978principles},
  \[\int_V\dc\beta=0.\] In particular, we can extend Lemma~\ref{lem: GL4.4} to the singular context, namely, for any nef (1,1) form $\theta$ with $\dc\theta^\ell=0$ for all $1\leq \ell\leq k $, then \[\widehat v_-(V,\theta|_V)=\int_V\theta^k.\]

\section{Null loci and Nakai--Moishezon criterions}\label{sect: null}
\subsection{Extension theorem}
{
 % \color{blue}
 We define the null locus of a big form, following the algebraic version in~\cite{nakamaye00-base-loci}.
\begin{definition}
    [Null locus] Let $\theta$ be a big (1,1) form. We define the {\em null locus} of $\theta$, denoted by $\Null(\theta)$, to be the set
   \[\Null(\theta)=\bigcup_{\theta|_V\, \text{is not big}}V, \] where the union is taken over all purely positive-dimensional irreducible analytic subvarieties $V\subset X$.
\end{definition}}
We describe the null locus with the following example.
\begin{example}
    Let $Y$ be a compact complex manifold of dimension $n$ equipped with a positive (1,1) form $\omega_Y$. Let $\pi:X\to Y$ be a  blow-up of a point $y\in Y$ with the exceptional divisor $E=\pi^{-1}(y)\simeq\mathbb{C}\mathbb{P}^{n-1}$. Let $\mathcal{O}_X(E)$ denote the line bundle associated with $E$ and $\Theta_h(E)$ the curvature form associated with a Hermitian metric $h$ on $\mathcal{O}_X(E)$. 
    % The equivalence class $[\theta]=[\pi^*\omega_Y]$ satisfies the nef property, and has self-intersection $\int_X\theta^n=\int_Y\omega_Y^n>0$.
 The form $\theta=\pi^*\omega_Y$ is semipositive, so it is nef.   It was shown (from, e.g.,~\cite[Lemma 3.5]{demailly2004numerical}) that there exists a smooth Hermitian metric $h$ on the line bundle $\mathcal{O}_X(E)$ with curvature $\Theta_h(E)$ such that $\pi^*\omega_Y-\varepsilon\Theta_h(E)$ is positive definite on $X$ for $\varepsilon$ sufficiently small. In particular, let $s$ be a section of $\mathcal{O}_X(E)$ vanishing along $E$ of order 1. Then for any $\varepsilon$ small, we have that $\pi^*\omega_Y+\varepsilon\dc\log|s|^2_h$ is a Hermitian current on $X$ that belongs to the $\ddbar$-class $[\theta]$. Obviously, one sees that the latter current has almost analytic singularities precisely along $E$. Since $\theta$ is positive on $X\backslash E$,  we also observe that 
    \[E=\bigcup_{\theta|_V\;\text{is not big}}V, \] where the union is taken over all positive-dimensional irreducible analytic subvarieties $V\subset X$.
\end{example}
% \begin{definition}
%     [null locus] Let $[\theta]$ be a nef class and $\widehat v_-(X,\theta)$. We define the {\em null locus} of $\theta$, denoted by $\Sigma$, to be the set
%    \[\Sigma=\bigcup_{\widehat v_-(v,\theta|_V)=0}V, \] where the union is taken over all positive dimensional irreducible analytic subvarieties $V\subset X$.
% \end{definition}

The above example motivates us to prove our main theorem in this section.
\begin{theorem}\label{thm: main}
    Let $X$ be a compact complex manifold of dimension $n$. Let $\theta$ be a nef and big (1,1) form. Then
    \[\NK(\theta)=\Null(\theta).\]
    % If $n=2$ and $\theta$ is merely big, then the result still holds.
\end{theorem}
% \begin{proof}
%     We show that $[\theta]$ contains a K\"ahler current, so $NK(\theta)\neq \varnothing$.
% \end{proof}

% The following extension theorem is crucial in the proof of the main theorem. 

The key ingredient in the proof of the theorem above is an extension-type result for a Hermitian current with almost analytic singularities in nef and big $\ddbar$-class on complex subvarieties, generalizing \cite[Theorem 3.2]{collins2015kahler}. 
% We mimic the main technical idea of~\cite[Theorem 3.2]{collins2015kahler} in which they used Hironaka's resolutions of singularities and  Richberg's gluing technique in order to glue psh functions with analytic singularities.  
\begin{theorem}\label{thm: extension-current}
    Let $(X,\omega_X)$ be a compact complex manifold of dimension $n$. Let $\theta$ be a smooth real (1,1) form on $X$, whose $\ddbar$-class is nef and big. Let $E=V\cup\cup_iY_i$ be an analytic closed subvariety of $X$ of pure dimension $k$, where $V$ and $Y_i$ are irreducible components of $E$, and $V$ a compact positive-dimensional complex submanifold of $X$. 
     Let $S=\theta+\dc F$ be a Hermitian current with almost analytic singularities precisely along $E$.
    Let $T=\theta|_{V}+\dc\f $ be a Hermitian current in the class $[\theta|_V]$ on $V$ with almost analytic singularities. Then there exists a Hermitian current with almost analytic singularities $\widetilde{T}$ on $X$ in the $\ddbar$-class $[\theta]$ such that 
    % $\widetilde{T}|_{V}=T$. In particular, 
    $\widetilde{T}$ is smooth in a neighborhood of the generic point of $V$.
    % If $\theta$ is merely big and $n=2$, then the result still holds.
\end{theorem}

When $X$ belongs to the Fujiki class $\mathcal{C}$ and $\theta$ is closed, the result was proved by Collins and Tosatti~\cite[Theorem 3.2]{collins2015kahler}, and later simplified by Darvas et al.~\cite[Theorem 6.1]{darvas2023transcendental}. We adapt their arguments to our setting to prove the corresponding result.
 \begin{proof}
     [Proof of Theorem~\ref{thm: extension-current}] 
    We divide the proof into several steps.
    % , which we briefly organize as follows. 
%     First, we lift our problem to a modification $\widetilde{X}$ of $X$ and $\widetilde{V}$ of $V$.
%     % ; precisely, we use the resolution of singularities, allowing us to reduce the singularities of $T$ to a simple normal crossing divisor on the strict transform $\widetilde{V}$
%     Next, we construct local extensions on a finite covering of $\widetilde V$ in $\widetilde X$.
%     We then use the technique of Richberg to glue the local potentials to construct a local extension of Hermitian current in a neighborhood of $\widetilde{V}$. The crucial point is that local potentials have the same singularity type on any non-empty intersection. Finally, after possibly reducing the Lelong
% numbers of the global Hermitian current $S$,
%     we can glue the Hermitian current on the neighborhood to $S$ to obtain the desired one.

\smallskip
 \textbf{Step 1:} {\it Lifting to the desingularization.}
Let $\mathcal{I}_T\subset\mathcal{O}_V$ be the coherent sheaf of ideals defining the singularities of $T$ on $V$. It follows from Hironaka's theorem that we can find a principlization of the sheaf of ideals $\mathcal{I}_T$ (cf.~\cite{hironaka77-actavietnam,wlodarczyk09-Resolution}), that is, a finite sequence
\[  p': V'=V_r\xrightarrow{p_r}V_{r-1}\rightarrow \cdots\rightarrow V_1\xrightarrow{p_1}V_0=V\]
of blow-ups $p_i$ with smooth centers $C_{i-1}\subset V_{i-1}$. 
Next, regarding the first center $C_0$ as a submanifold of $X$, we blow it up inside $X$ and iterate $r$ times to obtain a sequence 
\[ \pi:\widetilde{X}=X_r\xrightarrow{\pi_r}X_{r-1}\xrightarrow{\pi_{r-1}}\cdots \xrightarrow{\pi_2}X_1\xrightarrow{\pi_1}X_0=X\]
of blow-ups $\pi_i:X_i\to X_{i-1}$ with smooth centers $C_{i-1}\subset X_{i-1}$. 
% For each blow-up $\pi_i$, we have that the strict transform $\widetilde{V}_i$ of $V_{i-1}$ is smooth and isomorphic to $V_i$ and the restriction of $\pi_i$ to $\widetilde{V}_{i}$ composed with this isomorphism is exactly $p_i$. 
We observe that if $\widetilde V$ is the strict transform of $V$ under $\pi$, then $\widetilde V\to V'$ 
is a proper finite bimeromorphic morphism, so it is indeed an isomorphism since $V'$ is smooth (due to Zariski's main theorem).
We also see that the exceptional divisor $E_i$ of $\pi_i$ has simple normal crossings with $\widetilde{V}_i$ and $E_i\cap\widetilde{V}_i$ coincides with the exceptional divisor of $p_i$. Finally, on $\widetilde X$, we get a divisor $D=\sum D_\ell$ which has simple normal crossings with the strict transform $\widetilde{V}$ of $V$ via birational morphism $\pi$ and
\[(\pi|_{\widetilde V})^*\mathcal{I}_T=\mathcal{O}_{\widetilde{V}} \left(-\sum_\ell a_\ell D_\ell|_{\widetilde V} \right),\quad a_\ell>0. \]
 Note that since the analytic subvariety cut out by $\mathcal{I}_T$ has codimension at least 2 in $X$, it follows that all divisors $D_\ell$ are $\pi$-exceptional divisors. Let $\widetilde{Y}_i$ be the strict transform of $Y_i$ in $\widetilde{X}$. 

% By assumptions, we can find a current $T$ and $\varepsilon>0$ such that
% \[T=\theta|_V+\dc\f\geq 3\varepsilon\omega|_V\]
% in the sense of currents on $V$, where $\omega$ denotes a Hermitian metric on $X$. 
% Lifting to $\widetilde{X}$ we obtain
% \[\pi^*\theta+\dc(\f\circ\pi)\geq 3\varepsilon\pi^*\omega|_{\widetilde{V}}. \]
It follows from \cite[Lemma 3.5]{demailly2004numerical} that there is a small $\delta>0$ such that
\[\pi^*\omega_X+\delta\dc\log|s_D|^2_{h_D}\geq \widetilde\omega \]
for some Hermitian metric $\widetilde{\omega}$ on $\widetilde{X}$, where $s_D$ is a holomorphic section defining $D=\sum_\ell  D_\ell$ and $h_D$ is a smooth Hermitian metric on the line bundle $\mathcal{O}_{\widetilde X}(D)$.
Thanks to~\cite[Lemma 2.1, Theorem 3.2]{demailly2004numerical}, there exists a globally defined $C_0\widetilde{\omega}$-psh function $\psi$ on $\widetilde{X}$ for some $C_0>0$, with almost analytic singularities along $\cup_i\widetilde{Y}_i$. 

By assumption, we assume $T=\theta|_V+\dc\f\geq 4\varepsilon_0\omega_X|_V$ for some $\varepsilon_0>0$.
If we set $\widetilde{\f}\coloneqq\f\circ \pi+\frac{\varepsilon_0}{C_0}\psi+\delta\log|s_D|^2_{h_D}$, then $\widetilde\f$ has analytic singularities along $\cup_\ell D_\ell\cup_i\widetilde Y_i\cup \widetilde V$ and
\[ \pi^*\theta|_{\widetilde V}+\dc\widetilde\varphi\geq 3\varepsilon_0\widetilde{\omega}|_{\widetilde{V}}.\]
Similarly, the function $\widetilde F\coloneqq F+\delta\log|s_D|^2_{h_D}$ has analytic singularities along $\cup_\ell D_\ell\cup_i\widetilde Y_i\cup \widetilde V$. Up to shrinking $\varepsilon_0>0$, we assume that \[\pi^*\theta+\dc\widetilde F\geq3\varepsilon_0\widetilde\omega. \]
% In the next step, we will explain how to extend $\widetilde{\f}$ to a small neighborhood of $\widetilde{V}$.
% As follows from~\cite[Lemma 2.1, Theorem 3.2]{demailly2004numerical}, there exists a quasi-psh function $\widetilde F$ with almost analytic singularities such that $\widetilde F|_{\widetilde V}\equiv-\infty$. Up to shrinking $\varepsilon>0$, we can assume that $\pi^*\theta+\dc \widetilde F\geq3\varepsilon\widetilde\omega $ on $\widetilde X$.

     Since $\widetilde{V}$ is a closed submanifold of $\widetilde{X}$ and $\sum_\ell D_\ell|_{\widetilde{V}}$ has simple normal crossing support, we can find a finite open covering $\{W_j\}_{j=1}^N$ of $\widetilde{V}$ such that on each local chart $W_j$ there exist local coordinates $z_j=(z_j^1,\ldots z_j^n)$ with $\widetilde V\cap W_j=\{z_j^1=\cdots=z_j^{n-k}=0 \}$, where $k=\dim \widetilde{V}$, and with $\cup_\ell D_\ell\cap W_j=\{z^{i_1}_j\cdots z^{i_p}_j=0\}$, $n-k+1\leq i_1,\ldots,i_p\leq n$. On $W_j$, the functions $\widetilde\f$ and
     $\widetilde F$ can be written as
     \[\widetilde\f=\frac{\alpha}{2}\log\sum_{p=1}^{n-k}|z_j^p|^{2a_p}+\mathcal{C}^\infty ,\quad 
     \widetilde F=\frac{\beta}{2}\log\sum_{p=1}^{n-k}|z_j^p|^{2b_p}+\mathcal{C}^\infty. \]
Since $\pi^*\theta$ is nef, for any small $\lambda\in (0,1)$ there exists a smooth function $\rho_\lambda$ such that
\[\pi^*\theta+\dc\rho_\lambda\geq -3\lambda\varepsilon_0\widetilde\omega \]
holds on $\widetilde X$, and we can normalize $\rho_\lambda$ so that it is strictly positive. Then we have \[\pi^*\theta+\dc(\lambda\widetilde F+(1-\lambda)\rho_\lambda)\geq 3\varepsilon_0\lambda^2\widetilde\omega. \]
We choose $\lambda$ so small that $\lambda \beta b_p<\alpha a_p$ for any $p=\overline{1,n-k}$.  We stress that this is the only place where the nefness of $\theta$ is used.

 Set $F'\coloneqq\lambda\widetilde F+(1-\lambda)\rho_\lambda$. We take $0<\varepsilon<\lambda^2\varepsilon_0$ so that
\[ \pi^*\theta|_{\widetilde V}+\dc\widetilde\f\geq 3\varepsilon\widetilde\omega|_{\widetilde V},\;\text{and}\;\pi^*\theta+\dc F'\geq 3\varepsilon\widetilde\omega.\]
 
     \smallskip
 \textbf{Step 2:} {\it Gluing technique of Richberg.} In this step, we follow the trick of~\cite{darvas2023transcendental}.
     
    We denote by $z'_j=(z^1_j,\ldots,z^{n-k}_j)$ and $z''_j=(z^{n-k+1}_j,\ldots,z^n_j)$ on $W_j$.  We define a function $\f_j$ on $W_j$ by
\[ \f_j(z_j',z_j'')=\log\Big[e^{\widetilde\varphi(0,z_j'')}+e^{ F'(z',z'')} \Big]+A\cdot \textrm{dist}((z',z''),\widetilde{V})^2, \]
where $\textrm{dist}$ denotes the Riemannian distance with respect to $\widetilde\omega$ and $A>0$.
Up to shrinking the $W_j$'s if necessary, we can choose $A>0$ so large that $\widetilde\varphi(0,z''_j)+A\cdot \textrm{dist}((z',z''),V)^2$
  and $F'(z_j,z''_j)+A\cdot \textrm{dist}((z',z''),V)^2$ are $(\pi^*\theta-2\varepsilon\widetilde\omega)$-psh. Hence, 
 \[\pi^*\theta+\dc\f_j\geq 2\varepsilon\widetilde{\omega}\;\text{on}\; W_j \]
 holds for all $j$; see, e.g.,~\cite[I.4.16]{demaillycomplex}.
By construction, ${\varphi_j}$ is continuous on $W_j\setminus \widetilde V$ and $\f_j$ has almost analytic singularities on $W_j$. 
% Moreover, $e^{\f_j}$ is smooth on $W_j$, it follows from~\cite[Lemma 6.2]{darvas2023transcendental} that for any point $x\in W_j\cap W_\ell\cap\widetilde V$ and $\eta>0$ there exists an open neighborhood of $x$ such that $\f_j>\f_\ell-\eta$. 
\begin{claim}
    For any point $x\in W_j\cap W_\ell\cap\widetilde V$ and $\eta>0$ there exists an open neighborhood of $x$ such that $\f_j>\f_\ell-\eta$.
\end{claim}
\begin{proof}
    Since we have rescaled $\widetilde F$ to obtain $F'$, the H\"older exponent of $\widetilde\f$ is larger than that of $F'$. The claim follows the same arguments as in~\cite[Lemma 6.2]{darvas2023transcendental}.
\end{proof}
Fix slightly smaller open sets $W_j'\Subset U'_j\Subset U_j \Subset W_j$ such that $\widetilde V\subset \cup_jW_j'$.
Let $0\leq \chi_j\in\mathcal{C}^\infty(W_j)$ be a cutoff function such that $\textrm{supp}(\chi_j)\Subset W_j'^c$ and $\chi_j>1$ in a neighborhood of $\overline{U_j\backslash U'_j}$. We choose $\eta>0$ so small that $\dc \eta\chi_j\leq \varepsilon\widetilde\omega$. The function $\widetilde\f_j\coloneqq\f_j-\eta\chi_j$ satisfies $\pi^*\theta+\dc\widetilde\f_j\geq \varepsilon\widetilde\omega$ on $W_j$.
We set for $x\in \cup_jU_j$ 
\begin{equation*}
    \f'(x)\coloneqq\max_{1\leq\ell\leq N}\widetilde\f_\ell(x).
\end{equation*}
For each $x\in W_j'$, let $I_x$ denote the set of indices $\ell$ such that $x\in\overline{U_j\backslash U'_j}$. Fixing $x\in W_j'$, there exists an open neighborhood $B_x\subset W_j'$ such that for any $y\in B_x$ we have $I_y\subset I_x$. We can further shrink $B_x$ further such that $\chi_\ell>1$ on $B_x$ for all $\ell\in I_x$. Hence, on $ B_x$ \begin{equation}\label{eq: thm34} \widetilde\f_j=\f_j-\eta\chi_j=\f_j>\f_\ell-\eta>\f_\ell-\eta\chi_\ell,\;\; \forall\,\ell\in I_x.\end{equation}
Let $J_x\subset \{1,\ldots,N\}\backslash I_x$ denote the set of indices $\ell$ such that $U_\ell\cap B_x\neq \varnothing$. We observe that $B_x\subset U_\ell$ for all $\ell\in J_x$. Indeed, if this were not the case, then there would be $y\in B_x$ such that $y\in \partial U_\ell'\subset \overline{U_\ell\backslash U'_\ell}$. So $\ell\in I_y\subset I_x$, which is a contradiction.

Let $(B_{x_i})_{i=1}^r$ be a finite cover of $\widetilde V$ with $x_i\in W_{j_i}'\cap \widetilde V$. Set $W:=\cup_iB_{x_i}$.
As follows from~\eqref{eq: thm34}, we have on each $B_{x_i}$
\[\f'|_{B_{x_i}}=\max\bigg\{\widetilde\f_{j_i}|_{B_{x_i}}, \max_{\ell\in J_{x_i}}\widetilde \f_{\ell}|_{B_{x_i}} \bigg\}=\max\bigg\{\f_{j_i}|_{B_{x_i}}, \max_{\ell\in J_{x_i}}\widetilde \f_{\ell}|_{B_{x_i}} \bigg\}, \] 
 which satisfies $\pi^*\theta+\dc\f'\geq \varepsilon\widetilde\omega$ on $B_{x_i}$. Hence, $\pi^*\theta+\dc\f'\geq \varepsilon\widetilde\omega$ on $W$. Clearly, $\f'\geq\widetilde \f_j=\f_j$ on $W_j'$. Since $\widetilde\f_\ell=\f_\ell-\eta\chi_\ell\leq \f$ on $\widetilde V\cap U_j$ for all $\ell$ we have $\f'|_{\widetilde V}=\f$.  Up to shrinking $W$ slightly, we may assume that $\f'$ is defined on an open neighborhood $\widetilde W$ of $W$. 

\smallskip
     \textbf{Step 3:} {\it Completion of the proof.} 
     % From the previous step, there exist an open neighborhood $U$ of $\cup_\ell D_\ell \cup \cup_i\widetilde{Y}_i$ and $\lambda\in]0,1[$ so small such that $\lambda\widetilde \psi+(1-\lambda)\rho_\lambda>\f''$ on a neighborhood of $\partial W\cap U$. Since on $\partial W\backslash U$ we can find a constant $B>0$ so large that \[ \lambda\widetilde \psi+(1-\lambda)\rho_\lambda>\f''-B\]  hence $\lambda\widetilde \psi+(1-\lambda)\rho_\lambda>\f''-B$ in a neighborhood of $\partial W$.
   Let $C>0$ be such that $F'+C>\f'$ in a neighborhood of $\partial \widetilde W$. We can define a function $\Phi$ on $\widetilde X$ by
     \[\Phi=\begin{cases}
         \max(\f',F'+C)&\text{on}\; \widetilde W,\\
         F'+C& \text{on}\; \widetilde{X}\backslash \widetilde W.
     \end{cases} \]
     One easily sees that $ T'=\pi^*\theta+\dc{\Phi}\geq \varepsilon\widetilde{\omega}$ on $\widetilde X$. Observe that $\Phi=\f'$ in an open neighborhood of the generic point of $\widetilde V$ because $\widetilde{F}|_{\widetilde V}\equiv -\infty$. Therefore, $\Phi|_{\widetilde V}$ is smooth in an open neighborhood of the generic point of $\widetilde{V}$.
Pushing forward $T'$ on $X$, we obtain a Hermitian current $\widetilde T=\pi_* T'$ in the $\ddbar$-class $[\theta]$ which is an extension of $T$  such that the restriction $\widetilde{T}|_V$ is smooth in an open neighborhood of the generic point of ${V}$. The poof is complete.
 \end{proof}

\begin{proof}
    [Proof of Theorem~\ref{thm: main}] 
We first show that $\Null(\theta)\subset \NK(\theta)$.
{Let $x\in X$ and $V$ be an irreducible analytic subvariety of $X$ of positive dimension such that $\theta|_V$ is not big. Suppose by contradiction that $x$ does not belong to $\NK(\theta)$, so there is a Hermitian current $T\in [\theta]$ with almost analytic singularities and smooth in an open neighborhood of $x$. We write $T=\theta+\dc\f$ for some $\theta$-psh function $\f$. However, $\theta|_V+\dc\f|_V$ defines a Hermitian current with almost analytic singularities on $V$ because $\f|_V\not\equiv-\infty$, a contradiction.}

\smallskip
For the inverse implication $\Null(\theta)\supset \NK(\theta)$, 
suppose by contradiction that there exists a point $x\in \NK(\theta)\backslash \Null(\theta)$. Let $V$ be an irreducible component of the closed analytic set $\NK(\theta)$ containing $x$. Thanks to Lemma~\ref{lem: isolated}, the pure dimension of $V$ is strictly positive. 

Assume first that $V$ is smooth. {Since $\theta|_V$ is big}, we can find a Hermitian current $T=\theta|_V+\dc\f$ with almost analytic singularities (Theorem~\ref{thm: regularization}). We decompose
\[\NK(\theta)=V\cup\bigcup_{i\in I} Y_i,\] for disjoint irreducible components $V$, $Y_i$ of $\NK(\theta)$ which are all purely positive dimensional as follows from Lemma~\ref{lem: isolated}.   
We can thus apply Theorem~\ref{thm: extension-current} with $E=\NK(\theta)$, with the current $S=\theta+\dc \psi$ in Theorem~\ref{thm: boucksom317} to obtain a Hermitian current $\widetilde T=\theta+\dc\Phi$ in the $\ddbar$-class $[\theta]$ on $X$ such that $\widetilde T|_V=T$ and $T$ is smooth in a neighborhood of the generic point of $V$. In particular, we must have $V \nsubseteq \Z(\widetilde T)$, which is a contradiction since $V \subset \NK(\theta)\subset \Z(\widetilde T)$.

If $V$ is singular, we take an embedded desingularization of $V$, that is, $\pi:\widetilde X\to X$, given by a finite composition of blow-ups with smooth centers. The strict transform $\widetilde V$ is smooth and intersects the exceptional divisor with simple normal crossings. Since $\pi$ is bimeromorphic, $\pi^*\theta$ is nef and big on $\widetilde{X}$, has the same properties on $\widetilde V$.   Then there exists a Hermitian current $ S\in[\pi^*\theta]$ with almost analytic singularities along an analytic subset $\NK(\pi^*\theta)$. Thanks to~\cite[Proposition 2.5]{tosatti18-nakamaye}, we have $$\widetilde V\subset \NK(\pi^*\theta)=\textrm{Exc}(\pi)\cup \pi^{-1}(\NK(\theta)).$$
% , otherwise $S$ is smooth in a neighborhood of the generic point of $\widetilde V$, and then its push forward $\pi_* S$ is smooth near the generic point of $ V$, which contradicts $V\subset \NK(\theta)$.
We can apply Theorem~\ref{thm: extension-current} with $ S$ and a Hermitian current $T=\pi^*\theta|_{\widetilde V}+\dc\f$ to obtain a Hermitian current $\widetilde T=\pi^*\theta+\dc\widetilde\Phi$ in the $\ddbar$-class $[\pi^*\theta]$ such that $\widetilde T|_{\widetilde V}$ is smooth in a neighborhood of the generic point of $\widetilde V$. Hence, the direct image $\pi_*\widetilde T$ is smooth near the generic point of $V$, which is again a contradiction. 
\end{proof}
The following example tells us that the nef condition in Theorem~\ref{thm: main} is essential even in the algebraic version of Nakamaye's theorem; we thank Valentino Tosatti for pointing out this example.
\begin{example}
    Let $\pi_1 : X_1\to\mathbb{C}\mathbb{P}^2$ be a blow-up at one point, with exceptional divisor $F_1$. Let $\pi_2:X\to X_1$ be the blow-up of a point $y\in F_1$ with exceptional divisor $F_2$. The exceptional divisor $E$ of $\pi=\pi_2\circ \pi_1$ has two components $E=E_1+E_2$ where $E_1$ is the strict transform of $F_1$ under $\pi_2$ while $E_2=F_2$. We have $E_1^2=-2$, $E_2^2=-1$ and $E_1\cdot E_2=1$. We take $L=\pi^*H+E_1$ where $H $ is a hyperplane bundle on $\mathbb{C}\mathbb{P}^2$. It is clear that $L$ is big and that its positive part in the divisorial Zariski decomposition is $\pi^*H$, which is nef. Nakamaye's theorem yields
    \[ \mathrm{E_{nH}}(L)=\mathrm{E_{nH}}(\pi^*H)=E_1\cup E_2.\]
   However, since the restriction $L|_{E_2}$ is a line bundle on a (rational) curve with degree $E_1\cdot E_2=1$, it is ample. This implies $\NK(L)\neq\Null(L)$.
\end{example}

\subsection{Proof of Theorem~\ref{thmA}}
Under some additional assumptions on the manifold, we can determine the null locus as follows.

\begin{theorem}\label{thm: null=nb} Let $X$ be a $n$-dimensional compact complex manifold equipped with a Hermitian metric $\omega_X$. 
   Assume that $\omega_X$ satisfies the volume condition~\eqref{H}.
    Let $\theta$ be a nef (1,1) form.
    Then $\theta$ is big if and only if $\widehat v_-(X,\theta)>0$. Moreover 
   \begin{equation}\label{eq: null=nb}
      \NK(\theta)= \Null(\theta)=\bigcup_{\widehat v_-(V,\theta|_V)=0} V,
   \end{equation} 
     where the union is taken over all purely positive-dimensional irreducible analytic subvarieties of $V\subset X$.
\end{theorem}

\begin{proof} We denote by $N$ the set on the right-hand side of \eqref{eq: null=nb}.
    The first statement is shown in~\cite[Theorem 4.6]{guedj2022quasi2}. It suffices to show that $\NK(\theta)=N$.
    
    We first show that $N\subset \NK(\theta) $. Let $x\in X$ and $V$ an irreducible analytic subvariety of $X$ of dimension $k>0$ such that $\widehat v_-(V,\theta|_V)=0$. Suppose by contradiction that $x$ does not belong to $\NK(\theta)$, so there is a Hermitian current $T\in [\theta]$ with almost analytic singularities and smooth in an open neighborhood of $x$. We write $T=\theta+\dc\f$ for some $\theta$-psh function $\f$. 
Set $\psi=\f|_V$. 
% In particular, $\psi$ is smooth in a Zariski open set $U\subset V$, and the Monge-Ampere product $(\theta+\dc\psi)^k$ is well-defined on $V$.  
% Since $T$ is a smooth metric on $U$ we have
% \[ 0< \int_U(\theta+\dc\f)^k\leq \int_V(\theta+\dc\f)^k=\int_V\theta^k=0,\]
% which is a contradiction. 

Suppose first that $V$ is smooth.
 Assume that $\theta|_V+\dc\psi\geq \omega_V$ for some Hermitian metric $\omega_{V}$, and $\sup_V\psi=0$. Fix $\varepsilon>0$ and $u\in\PSH(V,\theta|_V+\varepsilon\omega_V)\cap \mathcal{C}^\infty(V)$. Set $v:=P_{\omega_V}(u-\psi)$.  We note that $u-\psi$ is merely well defined outside a pluripolar set and $v+\psi\leq u$ almost everywhere, hence everywhere.
 It follows from~\cite[Theorem 9.17]{guedj2017degenerate} that $v$ is bounded from above. By the classical pluipotential theory (see, e.g.,~\cite[Theorem B]{guedj2022quasi2}), $v$ is a bounded $\omega$-psh function and $(\omega_V+\dc v)^k$ is supported on the contact set $$D=\{v=u-\psi\}\subset \{\psi>-\infty\},$$ where the product $(\cdot)^k$ is understood as the Monge--Amp\`ere product in the sense of Bedford--Taylor; cf.~\cite[Theorem 2.3]{guedj2022quasi2}. By the maximum principle (Lemma~\ref{lem: max-princ}), we have
\[\mathbf{1}_{D\cap\{\psi>-\infty\}}(\theta|_V+\varepsilon\omega_V+\dc(v+\psi))^n\leq\mathbf{1}_{D\cap\{\psi>-\infty\}}(\theta|_V+\varepsilon\omega_V+\dc u)^n.  \]
Hence,
\[\int_V(\omega_V+\dc v)^n= \int_D(\omega_V+\dc v)^n\leq \int_V(\theta|_V+\varepsilon\omega_V+\dc u)^n \]
since $D\cap\{\psi>-\infty\}=D$. Therefore, $$0<v_-(V,\omega_V)\leq v_-(V,\theta|_V+\varepsilon\omega_V) $$ by assumption. Taking the infimum for $\varepsilon> 0$ gives a contradiction. 

If $V$ is singular, then we consider $\pi:X'\to X$ an embedded resolution of singularities obtained by a finite composition of blow-ups with smooth centers and assume that $V'$ is the strict transform of $V$.  We set $\psi'= (\f\circ \pi)|_{V'}+\delta\log|s|^2_h$ where $s$ is a section defining the exceptional divisor $\textrm{Exc}(\pi)$, and $\delta>0$ small enough. Then $\pi^*\theta|_{V'}+\dc\psi'\geq \omega_{V'}$ 
for some Hermitian metric $\omega_{V'}$. We apply the argument above with $\psi'$ and $\omega_{V'}$ to obtain
\[0<v_-(V',\omega_{V'})\leq v_-(V',\pi^*\theta+\varepsilon\omega_{V'}), \]
which gives a contradiction.

Conversely,  suppose by contradiction that there exists a point $x\in \NK(\theta)\backslash N$. Let $V$ be an irreducible component of the closed analytic set $\NK(\theta)$ containing $x$. Thanks to Lemma~\ref{lem: isolated}, we have that the pure dimension of $V$ is strictly positive, denoted by $k$. 

Assume first that $V$ is smooth. We see that $\widehat v_-(V,\theta|_V)>0$ because $x\notin N$. Then, by the first statement, we can find a Hermitian current $T=\theta|_V+\dc\f$ with almost analytic singularities (thanks to Demailly's regularization theorem). We decompose
\[\NK(\theta)=V\cup\bigcup_i Y_i,\] for disjoint irreducible components of $\NK(\theta)$, which are all purely positive dimensional, due to Lemma~\ref{lem: isolated}.   
We apply Theorem~\ref{thm: extension-current} with $E=\NK(\theta)$, with the current $S=\theta+\dc \psi$ in Theorem~\ref{thm: boucksom317} to obtain a Hermitian current $\widetilde T=\theta+\dc\Phi$ in the $\ddbar$-class $[\theta]$ on $X$ satisfying that $\widetilde T|_V=T$ and $T$ is smooth in a neighborhood of the generic point of $V$. In particular, we must have $V \nsubseteq \Z(\widetilde T)$, which is a contradiction because  $$V \subset \NK(\theta)\subset \Z(\widetilde T).$$
If $V$ is not smooth, we argue the same as in the proof of Theorem~\ref{thm: main} to obtain again a contradiction; see also~\cite[page 1180]{collins2015kahler}.
\end{proof}
We end this section with some examples of manifolds $X$ equipped with a Hermitian metric $\omega_X$, which satisfies the volume assumption~\eqref{H}.
 \begin{example}
 Obviously, if $X$ is K\"ahler or belongs to the Fujiki class $\mathcal{C}$, then it always admits a metric $\omega_X$, which satisfies the condition~\eqref{H}.
      
       When $\dim_\mathbb{C}X=2$ or 3 and $\omega_X$ is pluriclosed, i.e., $\dc\omega_X=0$ , it has been shown in~\cite[Theorem 3.3]{angella2022plurisigned} that the condition~\eqref{H}  is satisfied.

        If $\omega_X$ satisfies the curvature condition given by Guan--Li, namely,
     \[ \dc\omega_X=0,\; {\rm d}\omega_X\wedge {\rm d^c}\omega_X=0,\]
     which is equivalent to $\dc\omega_X^k=0$ for $k=1,\dots,n-1$, then we have 
     \[ \int_V(\omega_X+\dc\f)^{\dim V}=\int_V\omega_X^{\dim V},\,\forall\,\f\in\PSH(V,\omega_X|_V)\cap\mathcal{C}^\infty(X),\]
     for any compact submanifold (without boundary) of $X$. 
     We can precisely choose $X$ to be the product of a complex surface and a K\"ahler manifold. Namely, $X=M\times N$ where $M$ is a complex surface and $N$ is K\"ahler, and $\omega_X=p_1^*\omega_M+p_2^*\omega_N$ where $\omega_M$ is a $\dc$-closed metric (it always exists thanks to Gauduchon's theorem) and $\omega_N$ is K\"ahler metric. 
    \end{example}
\begin{example}
It was shown in~\cite[Theorem 3.7]{guedj2022quasi2} that the condition~\eqref{H} is bimeromorphic invariant. Therefore,
    we can find various compact complex manifolds with such conditions. For example, we assume that a compact complex manifold $X$ admits a Hermitian metric $\omega_X$ satisfying condition~\eqref{H} and $\pi:X'\to X$ a finite composition of blow-ups of $X$ with smooth centers, then $X'$ carries a Hermitian metric $\omega_{X'}$ with the same property.
\end{example}

\subsection{Nakai--Moishezon's criterion on some compact Hermitian manifolds}
As a direct application of Theorem~\ref{thmA}, we discuss the following Nakai--Moishezon criterion for some compact non-K\"ahler manifolds.
\begin{theorem}\label{thm: Nakai-Moishezon}
    Let $X$ be a $n$-dimensional compact complex manifold equipped with a Hermitian metric $\omega_X$. 
   Assume that $\omega_X$ satisfies the volume condition~\eqref{H}. Let $\theta$ be a nef (1,1) form. Then the $\ddbar$-class $[\theta]$ is Hermitian if and only if $\widehat v_-(V,\theta|_V)>0$ for all irreducible analytic subvarieties $V\subset X$.
\end{theorem}
\begin{proof}
    If the $\ddbar$ class $[\theta]$ is Hermitian, then we can find  a Hermitian metric $\omega$ such that $\omega=\theta+\dc \f$ for $\f\in\mathcal{C}^\infty(X)$, and $\omega\geq \delta\omega_X$ for some $\delta>0$. 
    If $V\subset X$ is a closed smooth manifold, then by ~\cite[Lemma 4.4]{guedj2022quasi2},
    \[\widehat v_-(V,\theta|_V)=v_-(V,\omega|_V)\geq \delta^n v_-(V,\omega_X|_V)>0. \]
If $V$ is not smooth, then thanks to Hironaka's theorem~\cite{hironaka77-actavietnam,wlodarczyk09-Resolution},
we can take $\pi:X'\to X$ to be an embedded resolution of singularities of $V$, obtained as a composition of blow-ups with smooth centers. Let $V'$ be the strict transform of $V$, which is now a closed smooth submanifold of $X'$. Thanks to~\cite[Theorem 3.7]{guedj2022quasi2}, the condition~\eqref{H} is preserved. That theorem also yields
$$v_-(V',\omega')\geq  \delta^n v_-(V',\pi^*\omega_X)>0.$$

    Conversely, with $V=X$ we observe that $\widehat v_-(X,\theta)>0$, hence it follows from Theorem~\ref{thm: current-anasing} that there exists a Hermitian current $T$ in the $\ddbar$ class $[\theta]$ with almost analytic singularities along $\NK(\theta)$
    . 
    We apply Theorem~\ref{thm: null=nb} to get that the non-Hermitian (null) locus
    \[ \NK(\theta)=\Null(\theta)=\bigcup_{\widehat v_-(V,\theta|_V)=0}V\] is empty. Therefore, the Hermitian current $T$ is, indeed, a Hermitian metric.
\end{proof}
Consequently, we provide an alternative proof of the Nakai--Moishezon criterion for compact non-K\"ahler surfaces~\cite{buchdahl00-NakaiMoishezon,lamari1999courants}. In higher dimensional case, it is just a consequence of the main result of \cite{chiose2016kahler} and \cite{popovici2016sufficients}.
We further show that it also holds for singular compact varieties.
 % We are grateful to V. Tosatti for pointing out this to us.
% \begin{corollary}\label{coro: Nakai-Moishezon}
%     Let $(X^n,\omega_X)$ be a compact complex manifold with $\dc\omega_X^k=0$ for any $k=1,\ldots,n-1$. Let $\theta$ be a smooth real (1,1) form such that
%     \begin{enumerate}
%     \item $\dc\theta^k=0$, for any $k=1,\dots, n-1$;
%         \item for any $k=1,\ldots, n$, $$\int_X\theta^k\wedge\omega_X^{n-k}>0;$$
%         \item  for every irreducible
%         subvariety $V\subset X$, 
%         $$\int_V\theta^k\wedge \omega_X^{\dim V-k}>0\,\quad,\forall\,k=1,\dots,\dim V .$$
%     \end{enumerate}Then the $\ddbar$-class $[\theta]$ is Hermitian.
% \end{corollary} The assumption (2) could be absorbed in the (3).
% In the case of complex surfaces, thanks to~\cite[Prop. 5]{Buchdahl1999compact} we can relax the assumption $(3)$: $\int_V\theta>0$ for every irreducible curve $V$ with negative self-intersection $V\cdot V<0$.

\begin{proof}[Proof of Theorem~\ref{thm: N-M criterion}] We first treat the case where $X$ is smooth.
Obviously, the metric $\omega_X$ satisfies the condition~\eqref{H}.
Thanks to Theorem~\ref{thm: Nakai-Moishezon},  it suffices to show that $\theta$ is nef. Indeed, if $\theta$ were nef, we would have $\widehat v_-(V,\theta|_V)=\int_V\theta^{\dim V}$ by (1). 
We choose a number $t>0$ so large that the $\ddbar$-class $[\theta+t\omega]$ is Hermitian, and we denote by $t_0$ the minimum of such $t$. It remains to prove that $t_0\leq 0$.
Suppose by contradiction that $t_0>0$. We see that $[\theta+t_0\omega]$ is automatically nef, but not Hermitian. 
% Indeed, if $[\theta+t_0\omega]$ is Hermitian so is $[\theta+(t_0-\varepsilon)\omega]$ for $\varepsilon>0$ small enough. 
Since $t_0>0$, the conditions (1) and (2) yield
    \[ \widehat v_-(X,\theta+t_0\omega_X)=\int_X(\theta+t_0\omega_X)^n>0.\]
    By Theorem~\ref{thm: guedjlu}, the $\ddbar$-class $[\theta+t_0\omega_X]$ is big. On the other hand, the condition (3) and Theorem~\ref{thm: Nakai-Moishezon} imply that the $\ddbar$-class $[\theta+t_0\omega_X]$ is Hermitian, a contradiction. Therefore, we must have $t_0\leq 0$ so $\theta$ is nef. 
    % The proof directly follows from Theorem~\ref{thm: Nakai-Moishezon}. 

    If $X$ is singular, Hironaka's theorem allows us to take $\pi\colon X'\to X$ to be an embedded resolution of singularities of $X$, obtained by finitely many blow-ups with smooth centers; let $D$ denote the exceptional divisor. On $X'$, $\omega_{X'}\coloneqq\pi^*\omega_X+\dc\delta\log|s_D|_h$ is a Hermitian metric for $\delta>0$ sufficiently small, which satisfies $\dc\omega_{X'}^k=0$ for all $1\leq k\leq n-1$. We then apply the previous argument using $\pi^*\theta$ and $\omega_{X'}$ to obtain a smooth function $\psi$ on $X'$ such that $\Theta=\pi^*\theta+\dc\psi$ is a Hermitian metric on $X'$. The direct image $\pi_*\Theta$ is a Hermitian current with singularities along $X_{\rm sing}$. To remove these singularities, we use an inductive strategy on dimension, following the techniques in~\cite[page 94]{collins-tosatti-singular} or~\cite[pages 18--22]{das-hacon-paun2022mmp} (which is used earlier in~\cite{demailly2004numerical}). We will only outline the main ideas here.
    
    {\it Induction.} Suppose the theorem holds for any compact normal analytic variety of pure dimension smaller than $n$ and that for any analytic subset $Y\subset X$, there exists an open neighborhood $W_Y$ of $Y$ and a smooth function $\rho_Y$ on $W_Y$ such that $\theta|_{W_Y}+\dc\rho_Y\geq \varepsilon\omega_X|_{W_Y}$ for some $\varepsilon>0$. 
    %The latter can be obtained from the same arguments in the proof of Theorem~\ref{thm: extension-current} or from the  references aforementioned. 
    
   Set $Y=X_{\rm sing}$. We may assume that $Y$ is irreducible (cf. \cite{demailly2004numerical} or \cite[pages 20-21]{das-hacon-paun2022mmp}).
   %From~\cite[page 21]{das-hacon-paun2022mmp} there exists a modification $p:X_1\to X$ such that $X_1$
    Let $p: Y' \to Y$ be a desingularization. Then we can find a Hermitian metric of the form $p^*\theta|_{Y'} + \dc \psi_Y$ on $Y'$, whose pushforward defines a Hermitian current on $Y$ with singularities along a closed analytic subset $Z \subset Y$ (note that $p: Y' \setminus p^{-1}(Z) \to Y \setminus Z$ is biholomorphic). By the induction hypothesis, there exists a neighborhood $W_Z$ of $Z$ and a smooth function $\rho_Z$ on $W_Z$ such that $\theta|_{W_Z}+\dc\rho_{Z}\geq \delta\omega_X|_{W_Z}$ for some $\delta>0$.  
    
    Using the regularized maximum function (cf.~\cite[page 43]{demaillycomplex}), we can then glue $p_*\psi_Y$ and $\rho_Z$ to obtain the desired global potential on $Y$. Indeed, we consider the regularized maximum function $\varphi_Y:=\max_\eta\{p_*\psi_Y,\rho_Z-C \}$, where $C>0$ is a large enough constant, such that $\varphi_Y=\pi_*\psi_Y$ near the boundary of $W_Z$ since $\psi_Y$ is smooth on the complement of $Z$. We clearly observe that
$\varphi_Y=\rho_Z-C$ in a neighborhood of $Z$, since $p_*\psi_Y$ equals $-\infty$ when restricted to $Z$. 
By the usual properties of the regularized maximum of two functions, we have $\theta|_Y+\dc\varphi_Y\geq \delta'\omega_X|_Y$ for some $\delta'>0$. 
Then, exists a neighborhood $Y\subset W_Y\subset X$ and an extension $\rho_Y\in\mathcal{C}^\infty(W_Y)$ of $\varphi_Y$ such that $\theta|_{W_Y}+\dc\rho_Y\geq \varepsilon\omega_X|_{W_Y}$ for some $\varepsilon>0$.
    
    Finally, we apply the same gluing process for $\pi_*\psi$ and $\rho_{X_{\rm sing}}$ to conclude the proof.
    % Hence, using the regularized maximum function; cf.~\cite[page 43]{demaillycomplex}, we can glue $\pi_*\psi_Y$ with $\rho_Z$    
    % to get the desired one. Finally, we apply with $\pi_*\psi$ and $\rho_{X_{\rm sing}}$ to conclude.
\end{proof}

\section{Finite time singularities of the Chern--Ricci flow}\label{sect: crf} 
% Let $X$ be a compact complex $n$-dimension manifold equipped with a Hermitian metric $\omega_0$. We denote by $d^c=\frac{\sqrt{-1}}{2}(\Bar{\partial}-\partial)$ so that $\dc=\sqrt{-1}\partial\Bar{\partial}$.
Let $(X,\omega_0)$ be a compact Hermitian manifold of dimension $n$.
We consider a solution $\omega=\omega(t)$ of the Chern--Ricci flow starting from $\omega_0$ that is given by a smooth family of Hermitian metrics satisfying 
\begin{equation}\label{crf}
    \begin{cases}
        \frac{\partial\omega}{\partial t}=-\ric(\omega),\; 0\leq t<T\\
       \omega(0)=\omega_0.
    \end{cases}
\end{equation}
One observes that, just as the K\"ahler--Ricci flow, the Chern--Ricci flow is equivalent to a parabolic equation for a scalar function. That is, solving a solution to the Chern--Ricci flow starting at $\omega_0$ boils down to solving the parabolic complex Monge--Amp\`ere equation for $\f=\f(t)$ with $t\in[0,T[$,
\begin{equation}
    \frac{\partial\f}{\partial t}=\log\frac{(\theta_t+\dc\f)^n}{\omega_0^n},\quad\theta_t+\dc\f>0,\quad \f(0)=0,
\end{equation}
where $\theta_t\coloneqq \omega_0- t\ric(\omega_0)$. Since $X$ is compact, by the standard parabolic theory,
the Chern--Ricci flow always admits a unique solution in some time interval $[0,T[$ for some $0< T\leq \infty$; see, e.g.,~\cite[Theorem 3.5]{tosatti2018kawa}.
As follows from~\cite{tosatti2015evolution}, the maximal existence time $T$ can be characterized by
\[ T=\sup\{t>0: \exists\,\psi\in\mathcal{C}^\infty(X)\,\text{with}\,\theta_t+\dc\psi>0\}.\]
When $\omega_0$ satisfies the Guan--Li condition $\dc\omega_0^k=0$ for all $1\leq k\leq n-1$, we can check that this condition is preserved by the Chern--Ricci flow. Moreover, $\dc\theta_t^k=0$ for all $1\leq k\leq n-1$.
Thanks to the Nakai--Moishezon criterion (Theorem~\ref{thm: N-M criterion}), we obtain a geometric characterization of the maximal existence time
\begin{equation*}
  \begin{split}
     T=\sup\bigg\{t>0:&\int_V\theta_t^k\wedge\omega_0^{\dim V-k}>0,\\ &\forall\,  \text{irreducible subvarieties} \, V\subset X, k=\overline{1,\dim V}
     \bigg\}.
      \end{split}
\end{equation*}
This extends the result \cite[Theorem 1.3]{tosatti2015evolution} to higher dimensions.

In this section, we study finite time singularities of the Chern--Ricci flow, proving Theorem~\ref{thm-parabolic}.
Suppose that $T<\infty$, so that the flow develops singularities at a finite time. 
Following~\cite{collins2015kahler} (see also~\cite{tosatti2022chern}), the set on which the singularities of the Chern--Ricci flow develop on $X$ (depending on the initial metric $\omega_0$) is defined by 
\begin{equation}
    \label{set: sing} 
    \begin{split}
        \Z(\omega(t))=X\backslash\{x\in X| U\ni x  ,\exists\,\omega_T\, &\text{Hermitian metric on}\, U \\
        & \text{s.t.}\, \omega(t)\xrightarrow{t\to T^-}\omega_T\,\text{in}\, \mathcal{C}^\infty_{\rm loc}(U)\}.
    \end{split}
\end{equation}
 Following Z. Zhang~\cite{Zhang10-scalar}, we 
 % observe that the set of singularities of the Chern--Ricci flow $\Sigma$ can be characterized using the scalar curvature of evolving metrics. We 
define the singularity formation set $\Sigma$ of the flow (which
depends on the initial metric $\omega_0$) to be the complement of the set of points $x\in X$ such that there exists an open neighborhood $U$ of $x$ and a constant $C>0$ with $|\textrm{Rm}(t)|_{\omega(t)}\leq C$  on $[0, T[\times U$, where $\textrm{Rm}(t)$  denotes the curvature tensor of $\omega(t)$. Also, we define the set $\Sigma'$ as the complement of the set of points $x\in X$ such that there exists an open neighborhood $U$ of $x$ a constant $C>0$ with $|R(t)|_{}\leq C$ on $[0,T[\times U$, where $R(t)$ is the scalar curvature of $\omega(t)$.

 Gill and Smith~\cite{gill-smith14-chernricci} showed that the singularity set of the flow is equal to the region on which the Chern scalar curvature blows up. The latter set is the singularity formation locus of the flow due to~\cite{enders11-typeSingularities}. We give two equivalent definitions of the singularity set $\Z(\omega(t))$.
\begin{theorem}\label{thm: sing-formation}Assume that the Chern--Ricci flow~\eqref{crf} exists on the maximal time interval $[0,T[$ with $T<\infty$. Then 
\[ \Z(\omega(t))=\Sigma=\Sigma'.\]
\end{theorem}

\begin{proof}
This was proved by M. Gill and K. Smith. However, it is unclear how the derivative $\partial_iQ_1$ in (3.9) of~\cite{gill-smith14-chernricci} vanishes at the maximum point of $Q_2$. 
We give the proof here,
mimicking their estimates with minor modifications.
\smallskip

   Since a uniform bound on the curvature tensor implies an upper bound on the scalar curvature, we have $\Sigma'\subset\Sigma$.
   If the metrics $\omega(t)$ converges smoothly to a Hermitian metric $\omega_U$ on $[0,T[\times U$ for some open subset $U$, then we have $|\textrm{Rm}(t)|_{\omega(t)}\leq C$, hence the inclusion $\Z(\omega(t))\subset \Sigma'$.
   It remains to show that if for any $x\in X$ with an open neighborhood $U$ of $x$ and $C>0$ satisfying $|R(t)|\leq C$ on $[0,T[\times U$, then on
a possibly smaller open neighborhood $U'$ of $x$ we have that $\omega(t)$ converges smoothly to a Hermitian metric $\omega_{U'}$ as $t\to T^-$. 
We note that most of estimates are local.
%We remark that most of the estimates from the proof of Theorem~\ref{thm: singularity-flow} are local and hence could be applied again here.

Fixing $x\in K\subset \subset U$ a compact subset, we will establish uniform estimates on $\omega(t)$ in $K$. It follows from the flow equation that
\[\frac{\partial \f}{\partial t}=\log\frac{\omega(t)^n}{\omega_0^n},\quad \frac{\partial}{\partial t}\left(\frac{\partial\f}{\partial t}\right)=-R. \]
Since $|R(t)|\leq C$ on $[0,T[\times U$, we have $|{\partial^2_{tt}\varphi}|\leq C$. Integrating in time implies that $|\partial_t\f|\leq C$ and $|\f|\leq C$ on $[0,T[\times U$. Therefore, the quantity $H\coloneqq t\partial_t\f-\f-nt$ is uniformly bounded on $[0,T[\times U$, also is $\partial_t H=-tR-n$, and
\[ \left( \frac{\partial}{\partial t}-\Delta_\omega\right)H=-\tr_\omega\omega_0.\]
We set on $[0,T[\times K$
\[ Q\coloneqq\log\tr_{\omega_0}\omega+AH+e^{H}\] for $A>0$ to be determined hereafter. The quantity $Q$ reaches its maximum at some point $(t_0,x_0)\in [0,T[\times K$. 
% We apply the same arguments as in Lemma~\ref{lem: C2estimate} with $-H$ instead of $u$ to obtain
% \[0\leq\left( \frac{\partial}{\partial t}-\Delta_\omega\right)Q=(C-A-e^H)\tr_\omega\omega_0+C  \]
% with a uniform $C>0$. Noticing that we assumed here $(\tr_\omega\omega_0)^2\geq e^{-H}C(A+1)^2$ at $(t_0,x_0)$, otherwise we are done. 
 It follows from~\cite[Proposition 3.1]{tosatti2015evolution} (also~\cite[(4.2)]{tosatti2015evolution}) that 
		\begin{equation}\label{eq: c2estimate--TW}
		    \left(\frac{\partial}{\partial t}-\Delta_{\omega}\right)\log\tr_{\omega_0}\omega\leq \frac{2}{(\tr_{\omega_0}\omega)^2}\textrm{Re}(g^{\Bar{q}k}(T_0)^p_{kp}\partial_{\Bar{q}}\tr_{\omega_0}\omega)+C\tr_\omega \omega_0,
		\end{equation}
		where $(T_0)^p_{kp}$ denote the torsion terms corresponding to $\omega_0$. At the maximum point $(t_0,x_0)$ of $Q$ we have $\partial_{\Bar{q}} Q=0$, hence \[\frac{1}{\tr_{\omega_0}\omega}\partial_{\Bar{q}}\tr_{\omega_0}\omega+A\partial_{\Bar{q}}H+e^{H}\partial_{\Bar{q}} H=0.\]
		Therefore, the Cauchy--Schwarz inequality yields
		\begin{equation}\label{eq: CS-sing-set}
		\begin{split}
		\left| \frac{2}{(\tr_{\omega_0}\omega)^2}\textrm{Re}(g^{\Bar{q}k}(T_0)^p_{kp}\partial_{\Bar{q}}\tr_{\omega_0}\omega)\right|&\leq \left| \frac{2}{\tr_{\omega_0}\omega}\textrm{Re}((A+e^{H})g^{\Bar{q}k}(T_0)^p_{kp}\partial_{\Bar{q}} H\right|\\
		&\leq e^{H}|\partial H|^2_\omega+C(A+1)^2e^{-H}\frac{\tr_\omega \omega_0}{(\tr_{\omega_0}\omega)^2}
		\end{split}
		\end{equation} for uniform $C>0$ only depending on the torsion term. From~\eqref{eq: c2estimate--TW} and~\eqref{eq: CS-sing-set}, we have at $(t_0,x_0)$,
\begin{equation}\label{eq: est46}
		\begin{split}
		0\leq \left(\frac{\partial}{\partial t}-\Delta_{\omega}\right)Q &\leq  C(A+1)^2e^{-H}\frac{\tr_\omega \omega_0}{(\tr_{\omega_0}\omega)^2}+C\tr_\omega \omega_0  -(A+e^{H})\tr_\omega\omega_0.
		\end{split}
		\end{equation}
  If $(\tr_{\omega_0}\omega)^2\leq C(A+1)^2e^{-H}$ at $(t_0,x_0)$, then
\[ Q\leq C-\frac{1}{2}H+AH+e^H\leq C_1,\]
  we are done. Otherwise,
we choose $A=C+1$, then it follows from~\eqref{eq: est46} that at $(t_0,x_0)$,
\[\tr_\omega\omega_0\leq C. \]
On the other hand, since $\log(\omega^n/\omega_0^n)=\partial_t\f\leq C$, hence
\[ \tr_{\omega_0}\omega\leq C(\tr_\omega\omega_0)^{n-1}\frac{\omega^n}{\omega_0^n}\leq C.\]
This implies that $Q$, and thus $\tr_{\omega_0}\omega$ is bounded from above on $[0,T[\times K$. Finally, the local higher-order estimates for $\omega$ on $[0,T[ \times U'$ with $U'\subset K$
follow, thanks to Sherman--Weinkove's estimates~\cite{sherman-weinkove13-estimates}. The proof is complete.
\end{proof}
Consequently, the scalar curvature blows up at a finite time singularity, generalizing Zhang's analogous result~\cite{Zhang10-scalar} in the K\"ahler case.
\begin{corollary}[{\cite[Theorem 1.1]{gill-smith14-chernricci}}]
    If the Chern--Ricci flow develops a finite time singularity, then
the singularity formation set $\Sigma$ is nonempty, and furthermore, \[\limsup_{t\to T} \left(\sup_X R(t)\right)=+\infty.\]
\end{corollary}
\begin{proof}
    If $\Sigma=\varnothing$, then by Theorem~\ref{thm: sing-formation}, replacing $K$ with $X$, we would see that $\omega(t)$ converges in $\mathcal{C}^\infty$ to a Hermitian metric $\omega_T$ on $X$, a contradiction. By~\cite[Lemma 3.1]{gill-smith14-chernricci} and the maximum principle, the scalar curvature $R(t)$ is uniformly bounded from below. 
    The blow-up of the supremum of the scalar curvature again follows from Theorem~\ref{thm: sing-formation}.
\end{proof}
Next, we compare the singularity formation set of the flow $\Z(\omega(t))$ with the null locus of the limiting form $\theta_T=\omega_0-T\ric(\omega_0)$ in some specific contexts.
We see that the $\ddbar$ class $[\theta_T]$ is nef,
but not Hermitian. Since the definition of $\widehat v_-(X,\theta_T)$ does not depend on the choice of a reference Hermitian metric $\omega_X$ and for any $\varepsilon>0$,
\[(1-\varepsilon)\theta_T+\varepsilon\omega_0=\theta_{(1-\varepsilon){T}}, \]
so we have
\[ \widehat v_-(V,\theta_T|_V)=\inf_{t\to T^-}v_-(V,\omega(t)|_V)\]
for all irreducible analytic subvarieties  $V\subset X$.
% We say that the flow is {\em collapsing} at time $T$ if the volume of $X$ with respect to $\omega(t)$, denoted by $\Vol(X,\omega(t))$, goes to zero as $t\to T^-$. Otherwise, the flow is {\em non-collapsing} if $\Vol(X,\omega(t))\to 0$ as $t\to T$. 
We define the set-theoretic union of all such subvarieties for which the infimum volume shrinks to zero as $t\to T^-$,
\begin{equation}\label{def: union--zeroset}
 N(\omega(t))\coloneqq\bigcup_{v_-(V,\omega(t)|_V)\xrightarrow{t\to T^-
 }0}V.   
\end{equation} 
In particular, the set $N(\omega(t))$ is a closed analytic subset of $X$. 
% Thanks to Theorem~\ref{thmA} $N(\omega(t))=\NK(\theta_T)$ provided a Hermitian metric $\omega_X$ satisfying the condition~\eqref{H}. 

If $\theta_T$ is big, then we can find a $\theta_T$-psh function $\rho$ with almost analytic singularities along $\NK(\theta_T)$ such that $$\theta_T+\dc\rho\geq 2\delta_0\omega_0$$ for some $\delta_0>0$. Up to subtracting a positive constant, we can assume that $\sup_X\rho\leq 0$. We have
\[ \theta_t+\dc\rho=\frac{1}{T}\left((T-t)(\omega_0+\dc\rho)+t(\theta_T+\dc\rho)\right)\geq \delta_0\omega_0, \]
if $t\in [T-\varepsilon,T[$ for some $\varepsilon>0$.

We are interested in the behavior of the metrics $\omega(t)$ as $t\to T^-$,
proving the following theorem.
\begin{theorem}\label{thm: singularity-flow}
    Assume that the limiting form $\theta_T$ is big. Then, there exists a positive (1,1) current $\omega_T$ in $X$ in the $\ddbar$ class $[\theta_T]$, which is precisely smooth away from $\NK(\theta_T)$ and has minimal singularities, such that as $t\to T^-$ the metrics $\omega(t)$ converge to $\omega_T$ in both the current and the $\mathcal{C}_{\rm loc}^\infty(X\backslash \NK(\theta_T))$ senses.  In particular, $\Z(\omega(t))\subset              
    \NK(\theta_T)$.
\end{theorem}

\begin{lemma}\label{lem: bound}
    There is a uniform constant $C_0>0$ such that on $[0,T[\times X$ we have
		\begin{enumerate}[label=(\roman*)]
			\item $C_0\geq \f\geq \rho-C_0$;
			\item $C_0\geq \partial_t{\f}\geq C_0\rho-C_0$.
		\end{enumerate}
\end{lemma}

\begin{proof}
   The proof follows almost verbatim from the K\"ahler case~\cite[Section 4]{collins2015kahler}. We also refer to~\cite[Section 2]{tosatti2013chern}
   or~\cite[Section 5]{dang2023singularities} for a proof.
\end{proof}
\begin{lemma}\label{lem: C2estimate}
    There is a constant $C>0$ such that on $[0,T[\times X$ we have
    \[ \tr_{\omega_0}\omega\leq Ce^{-C\rho}.\]
\end{lemma}
\begin{proof}
    The proof relies on the trick of Phong--Sturm~\cite{phong2010dirichlet}, we repeat the estimates derived in the proof of~\cite[Lemma 4.1]{tosatti2015evolution} with minor modifications. Set $\psi\coloneqq\f-\rho+C_0\geq 0$, so  $e^{-\psi}\in ]0,1]$.
		We compute the evolution of \[Q\coloneqq \log\tr_{\omega_0}\omega-A\psi+e^{-\psi}\]
		for $A>0$ to be determined hereafter. 
		It suffices to show that $Q$ is uniformly bounded from above. We observe that $Q\leq C$ on $[0,T-\varepsilon]\times X$ for a uniform $C>0$. Fixing $T-\varepsilon<T'<T$, suppose that $Q$ attains its maximum at some point $(t_0,x_0)\in  [0,T']\times X$ with $t\in ]T -\varepsilon,T']$. In what follows, we compute at this point. 
  % From~\cite[Prop. 3.1]{tosatti2015evolution} (also~\cite[(4.2)]{tosatti2015evolution}) we have
		% \begin{equation}\label{eq: c2estimate--TW}
		%     \left(\frac{\partial}{\partial t}-\Delta_{\omega}\right)\log\tr_{\omega_0}\omega\leq \frac{2}{(\tr_{\omega_0}\omega)^2}\textrm{Re}(g^{\Bar{q}k}(T_0)^p_{kp}\partial_{\Bar{q}}\tr_{\omega_0}\omega)+C\tr_\omega \omega_0,
		% \end{equation}
		% where $(T_0)^p_{kp}$ denote the torsion terms corresponding to $\omega_0$. 
  At the maximum point $(t_0,x_0)$ of $Q$ we have $\partial_iQ=0$ hence \[\frac{1}{\tr_{\omega_0}\omega}\partial_i\tr_{\omega_0}\omega-A\partial_i\psi-e^{-\psi}\partial_i\psi=0.\]
		Therefore, the Cauchy--Schwarz inequality yields
		\begin{equation}\label{eq: c2estimate-CS}
		\begin{split}
		\left| \frac{2}{(\tr_{\omega_0}\omega)^2}\textrm{Re}(g^{\Bar{q}k}(T_0)^p_{kp}\partial_{\Bar{q}}\tr_{\omega_0}\omega)\right|&\leq \left| \frac{2}{\tr_{\omega_0}\omega}\textrm{Re}((A+e^{-\psi})g^{\Bar{q}k}(T_0)^p_{kp}\partial_{\Bar{q}}\psi\right|\\
		&\leq e^{-\psi}|\partial\psi|^2_\omega+C(A+1)^2e^\psi\frac{\tr_\omega \omega_0}{(\tr_{\omega_0}\omega)^2}
		\end{split}
		\end{equation} for uniform $C>0$ only depending on the torsion term.
		It thus follows from~\eqref{eq: c2estimate--TW} and~\eqref{eq: c2estimate-CS} that, at $(t_0,x_0)$,
		\begin{equation}\label{eq: c2Q}
		\begin{split}
		0\leq \left(\frac{\partial}{\partial t}-\Delta_{\omega}\right)Q &\leq C(A+1)^2e^\psi\frac{\tr_\omega \omega_0}{(\tr_{\omega_0}\omega)^2}+C\tr_\omega \omega_0\\
		&\quad -(A+e^{-\psi})\partial_t{\f} +(A+e^{-\psi})\tr_\omega(\omega -(\theta_t+\dc\rho))\\
		&\leq C(A+1)^2e^\psi \frac{\tr_\omega \omega_0}{(\tr_{\omega_0}\omega)^2}+(C-A\delta_0)\tr_\omega \omega_0+(A+1)\log\frac{\omega_0^n}{\omega^n},
		\end{split}
		\end{equation} where we have used $\theta_t+\dc\rho\geq \delta_0\omega_0$. If at $(t_0,x_0)$,  $(\tr_{\omega_0}   \omega)^2\leq e^\psi C(A+1)^2$  then at the same point we obtain \[Q\leq C+\frac{1}{2}\psi-A\psi+e^{-\psi}\leq C+1\] since $\psi\geq 0$, we are done.
		Otherwise, we choose $A=\delta_0^{-1}(C+2)$. Hence, from~\eqref{eq: c2Q} one gets at the maximum of $Q$,
		\[\tr_{\omega}\omega_0\leq C\log\frac{\omega_0^n}{\omega^n}+C,\]
		so
		\[ \tr_{\omega_0}   \omega\leq n(\tr_\omega\omega_0)^{n-1}\frac{\omega^n}{\omega_0^n}\leq C\frac{\omega^n}{\omega_0^n}\left(\log\frac{\omega_0^n}{\omega^n}\right)^{n-1}+C\leq C'\]since $\omega^n/\omega_0^n\leq C_0$ by Lemma~\ref{lem: bound} and $y\mapsto y|\log y|^{n-1}$ is  bounded from above as $y\to 0$. Thanks to Lemma~\ref{lem: bound} $(iii)$, $Q$ is bounded from above at its maximum,  finishing the proof.
\end{proof}
\begin{proof}[Proof of Theorem~\ref{thm: singularity-flow}]
We apply Lemmas~\ref{lem: bound} and~\ref{lem: C2estimate} to obtain that for any compact set $K\subset X\backslash\NK(\theta_T)$, $k\in\mathbb{N}$, there is a constant $C_{K,k}>0$ such that 
    \[ \|\f(t)\|_{\mathcal{C}^k(K)}\leq C_{K,k},\]
     using the complex parabolic Evans--Krylov theory together with Schauder’s estimate; see also Sherman--Weinkove's local curvature estimates~\cite{sherman-weinkove13-estimates}. This implies that for some constant $C_K>0$, \[ \frac{\partial}{\partial t}(\f(t)+C_Kt)\geq 0,\]
     hence $\f(t)+C_Kt$ is increasing in $t$ as well as being bounded from above. This yields that $\f(t)$ has a limit as $t\to T^-$, denoted by $\f_T$. An elementary argument shows that $\f(t)$ converges in $\mathcal{C}^\infty$ on compact subsets of $X\backslash\NK(\theta_T)$. By weak compactness, $\f_T$ extends to a $\theta_T$-psh function on $X$ and
     $\f(t)$ converges to $\f_T$ in $L^1(X)$. The current $\omega_T$ can be expressed as $\omega_T=\theta_T+\dc\f_T$. 
    % \[C_K^{-1}\omega_0\leq \omega(t)\leq C_K\omega_0\;\text{on}\,[0,T)\times K.\]
    % We then apply the local higher order estimates of Sherman--Weinkove~\cite{sherman-weinkove13-estimates}  to obtain uniform $\mathcal{C}^\infty$ estimates for $\omega(t)$ on any compact subset $K$ of $X\backslash\NK(\theta_T)$. 

    Next, we show that $\f_T$ has minimal singularities, that is, for any $\theta_t$-psh function $u$ there is a constant $C>0$ such that \[\f\geq u-C\quad\text{on}\, [0,T[\times X. \]
    Thanks to Demailly's regularization theorem (Theorem~\ref{thm: regularization}), we can find functions with logarithmic singularities $u_m$ decreasing to $u$ such that
    \[\theta_T+\dc u_m\geq -\varepsilon_m\omega_0,\quad\varepsilon_m\searrow 0. \]
We
set $$G=(T-t)\partial_t\f+\f+nt-\varepsilon_m(t\partial_t \f-\f-nt)-u_m.$$
On $X\backslash\Z(u_m)$, $G$ is smooth and
\[ \left(\frac{\partial}{\partial t}-\Delta_\omega \right)G\coloneqq\tr_{\omega}(\theta_T+\varepsilon_m\omega_0+\dc u_m)\geq 0.\]
Since $\sup_X u_m\leq C$ uniformly, so $G$ achieves its minimum in $ X\backslash \Z(u_m)$. Hence, the minimum principle implies that $G\geq -C$ independent of $m$, so
\[ (T-t-\varepsilon_m)\partial_t \f+(1+\varepsilon_m)\f(t)\geq u_m-C\geq u-C.\]
We take $m\to+\infty$ and use the fact that $\partial_t\f \leq C_0$ to conclude.
\end{proof}

\begin{remark}
   The current $\omega_T$ has minimal singularities but is not a Hermitian current. In fact, we write $\omega_T=\theta_T+\dc \varphi_T$. If $\omega_T$ were Hermitian, we would find that the current $\theta_T+\dc (1-\varepsilon)\varphi_T$ is Hermitian for $\varepsilon>0$ sufficiently small. By minimality, we have $\varphi_T+C \geq (1-\varepsilon)\varphi_T$ for a constant $C>0$ depending on $\varphi_T$, so $\varphi_T\geq -C\varepsilon^{-1}$ is bounded. This implies that $\Z(\f_T)=\varnothing$, hence the $\ddbar$-class $[\theta_T]$ is Hermitian, a contradiction.
\end{remark}
% \begin{proof}[Proof of Theorem~\ref{thm-parabolic}]
%     By assumption, $v_+(X,\omega_X)<+\infty$, $\widehat v_-(X,\theta_T)>0$, since $\theta_T$ is nef, hence $\theta_T$ is big. By Theorem~\ref{thm: singularity-flow}, there are no singularities developing outside of $\NK(\theta_T)$. Theorem~\ref{thmA} shows that $\NK(\theta)=N(\omega(t))$, so it suffices to show that the metrics $\omega(t)$ developed singularities along $N(\omega(t))$. Suppose by contradiction that there is a point $x\in N(\omega(t))$ which does not belong to the singularity set $\Z(\omega(t))$.
% \end{proof}

\begin{theorem}
    Let $(X,\omega_0)$ be a compact Hermitian manifold of dimension $n$. Let $\omega(t)$ be a solution of the Chern--Ricci flow starting at $\omega_0$, which develops
a singularity at finite time $T$. Assume that the limiting form $\theta_T=\omega_0-T\ric(\omega_0)$ satisfies $\widehat v_-(X,\theta_T)>0$.
Then if either 
\begin{itemize}
    \item $n=2$,
   \item  or $n\geq 3$, $\dc\omega_0=0$, ${\rm d}\omega_0\wedge {\rm d^c}\omega_0=0$, 
\end{itemize}
then the metrics $\omega(t)$ develop singularities precisely along $N(\omega(t))=\NK(\theta_T)$.
\end{theorem}
Recall here that the set $N(\omega(t))$ is defined by~\eqref{def: union--zeroset}. We will see in the proof that $N(\omega(t))$ is the set-theoretic union of all such subvarieties for which the volume shrinks to zero as $t\to T^-$. This, therefore, partially answers the question of Tosatti--Weinkove~\cite[Question 6.1]{tosatti2022chern}.
\begin{proof} By assumption, the metric $\omega_0$ always satisfies the condition~\eqref{H}. Since $\theta_T$ is nef and $\widehat v_-(X,\theta_T)>0$, so $\theta_T$ is big due to Theorem~\ref{thm: guedjlu}.
  Thanks to Theorems~\ref{thm: singularity-flow} and~\ref{thm: null=nb},  we have $$\Z(\omega(t))\subset \NK(\theta_T)=N(\omega(t)).$$ It suffices to show that the metrics $\omega(t)$ develop singularities along $N(\omega(t))$. Suppose by contradiction that there is a point $x\in N(\omega(t))$ with a neighborhood $U\ni x$ in which, as $t\to T^-$, the metrics $\omega(t)$ smoothly converge to a Hermitian metric $\omega_U$ and there is a positive-dimensional irreducible subvariety $V\subset X$ containing $x$ and $v_-(V,\omega(t)|_V)\to 0$ as $t\to T^-$. 
    We divide it into two cases.

    If $n=2$, we must have $\dim V=1$ because $\widehat v_-(X,\theta_T)>0$. Then
    \[v_-(V,\omega(t)|_V)=\int_V\omega(t)\geq \int_{V\cap U}\omega(t)\xrightarrow{t\to T^-} \int_{V\cap U}\omega_U>0, \]
  a contradiction.

    If $n\geq 3$ and $\dc\omega_0=0$ and ${\rm d}\omega_0\wedge {\rm d^c}\omega_0=0$ then $\dc\omega_0^k=0$, $\forall\, 1\leq k\leq n-1$. The latter condition is preserved along the Chern--Ricci flow, i.e., $\dc\omega(t)^k=0$ for all $1\leq k\leq n-1$. Then we have $$v_-(V,\omega(t)|_V)=\int_V\omega(t)^{\dim V},$$ for every irreducible subvariety $V\subset X$, due to Lemma~\ref{lem: GL4.4}. On the other hand, 
    \[\int_V\omega(t)^{\dim V}\geq \int_{V\cap U}\omega(t)^{\dim V}\xrightarrow{t\to T^-} \int_{V\cap U}\omega_U^{\dim V}>0, \]
 a contradiction.
\end{proof}
In the above setup, 
we show that the Chern--Ricci flow performs a canonical surgical contraction.
\begin{theorem} Let $X$ be compact Hermitian manifolds of dimension $n$, equipped with a Hermitian metric $\omega_0$ with $\dc\omega_0=0$ and ${\rm d}\omega_0\wedge {\rm d^c}\omega_0=0$.   Assume that the solution $\omega(t)$ of the Chern--Ricci flow~\eqref{crf} starting at $\omega_0$ exists on $[0,T[$ with $T<\infty$ and the limiting form $\theta_T=\omega_0-T\ric(\omega_0)$ satisfies $\int_X\theta_T^n>0$. 
Assume there exists a map
    $\pi:X\to Y$ blowing down disjoint irreducible exceptional divisors $E_i$ on $X$ to points $y_i\in Y$ for $i=1,\dots,p$, with $Y$ a compact Hermitian manifold, such that there exist a function $f\in\mathcal{C}^\infty(X,\mathbb{R})$ and a smooth real (1,1)-form $\eta$ on $Y$ with
    \begin{equation*}\label{asterisk} \tag{$\ast$}
        \theta_T+\dc f=\pi^*\eta.
    \end{equation*}
    Then we have the following 
    \begin{enumerate}[label=(\roman*)]
        \item As $t\to T^-$, the metrics $\omega(t)$ converge to a smooth Hermitian metric $\omega_T$ on $X^\circ:=X\backslash \bigcup_{i=1}^p E_i$ in $\mathcal{C}^\infty_{\rm loc}(X^\circ)$. 
        Using $\pi$,  $\omega_T$ can be regarded as a Hermitian metric on $Y^\circ:=Y\backslash \{y_1,\dots,y_p\}$.
        \item  There exists a unique metric $d_T$ on $Y$ such that $(Y,d_T)$ is a compact metric space. 
        %Let $d_{\omega_T}$ be the distance function induced by $\omega_T$ on $Y^\circ$. Then there exists a unique metric $d_T$ on $Y$ extending $d_{\omega_T}$ to be zero on $\{y_1,\dots,y_p\}$ such that $(Y,d_T)$ is a compact metric space homeomorphic to $Y$. 
        % and $(Y,d_T)$ is the metric completion of $(Y^\circ, d_{\omega_T})$.
        \item  The metric space $(X,\omega(t))\rightarrow (Y,d_T)$ as $t\to T^-$ in the Gromov--Hausdoff sense.
        \item  A smooth maximal solution $g(t)$ of the Chern--Ricci flow on $Y$ starting at $\omega_T$ exists for $t\in]T,T_Y[$ with $T<T_Y\leq\infty$, and $g(t)$ converges to $\omega_T$ in $\mathcal{C}^\infty_{\rm loc}(Y^\circ)$ as $t\to T^+$.  
        \item The metric space $(Y,g(t))\rightarrow (Y,d_T)$  as $t\to T^+$ in the Gromov--Hausdorff sense.
    \end{enumerate}
\end{theorem}
\begin{proof}
    We observe that $\dc\omega_0^k=0$ for all $1\leq k\leq n-1$, this property is preserved along the Chern--Ricci flow. Since $\theta_T$ is nef and $\int_X\theta_T^n>0$ we have that $\theta_T$ is big. By the previous analysis, together with assumptions, we know that $$\NK(\theta_T)=\Null(\theta_T)=\bigcup_{1\leq i\leq p}E_i.$$
    Then $(i)$ follows immediately from Theorem~\ref{thm: singularity-flow}.

    We show that in the condition~\eqref{asterisk} the form $\eta$ can be replaced with a Hermitian metric $\omega_Y$ on $Y$. 
    Indeed, it is obvious that $\eta$ is nef and
% and \[\int_Y\eta^n=\int_X(\pi^*\eta)^n=\int_X\theta_T^n>0 \] so it is big. 
    for any irreducible subvariety $V\subset Y$ with $k=\dim V>0$, we have that  \[\int_V\eta^{k }=\int_{\pi^{-1}(V)}(\pi^*\eta)^k=\int_{\pi^{-1}(V)}\theta_T^k>0. \]
    Thanks to Nakai--Moishezon's criterion (Theorem~\ref{thm: Nakai-Moishezon}) the $\ddbar$-class $[\eta]$ is Hermitian on $Y$. We thus obtain $(ii)$ and $(iii)$ as follows from~\cite[Theorem 1.3]{tosatti2013chern}. 

    We have seen that $\omega_T$ is a positive current with minimal singularities, so it has bounded potentials.
   The push forward $\pi_*\omega_T$ is therefore a positive current on $Y$ with bounded potentials. In this case, we have $\pi_*\omega_T=\omega_Y+\dc\psi_T$ where $\psi_T$ is a bounded $\omega_Y$-psh function, which is smooth in $Y^\circ$. With this condition, the $(iv)$ and $(v)$ was proved by Nie~\cite{nie2017weak} and T\^o~\cite{to2018regularizing} independently.
\end{proof}
The major problems are whether the metric space $(Y,d_T)$ is the metric completion of $(Y^\circ,d_{\omega_T})$, and we can remove the condition~\eqref{asterisk}. Those may require new techniques. 
\bibliographystyle{alpha}
	\bibliography{bibfile}	

@article {boucksom2004divisorial,
    AUTHOR = {Boucksom, S.},
     TITLE = {Divisorial {Z}ariski decompositions on compact complex
              manifolds},
   JOURNAL = {Ann. Sci. \'{E}cole Norm. Sup. (4)},
  FJOURNAL = {Annales Scientifiques de l'\'{E}cole Normale Sup\'{e}rieure. Quatri\`eme
              S\'{e}rie},
    VOLUME = {37},
      YEAR = {2004},
    NUMBER = {1},
     PAGES = {45--76},
      ISSN = {0012-9593},
   MRCLASS = {32J18 (32C30)},
  MRNUMBER = {2050205},
MRREVIEWER = {Adam Gregory Harris},
       DOI = {10.1016/j.ansens.2003.04.002},
       URL = {https://doi.org/10.1016/j.ansens.2003.04.002},
}

@article{BoucksomGuedjLu2025-volume,
  title={{Volumes of Bott–Chern Classes}},
  author={Boucksom, S. and Guedj, V. and Lu, C.H. },
  journal={Peking Math J},
  year={2025}
}

@article {collins2015kahler,
    AUTHOR = {Collins, T. C. and Tosatti, V.},
     TITLE = {K\"{a}hler currents and null loci},
   JOURNAL = {Invent. Math.},
  FJOURNAL = {Inventiones Mathematicae},
    VOLUME = {202},
      YEAR = {2015},
    NUMBER = {3},
     PAGES = {1167--1198},
      ISSN = {0020-9910},
   MRCLASS = {32J25 (32Q15 32U40 53C55)},
  MRNUMBER = {3425388},
MRREVIEWER = {Jian Xiao},
       DOI = {10.1007/s00222-015-0585-9},
       URL = {https://doi.org/10.1007/s00222-015-0585-9},
}

@article {demailly1985measures,
    AUTHOR = {Demailly, J.-P.},
     TITLE = {Mesures de {M}onge-{A}mp\`ere et caract\'{e}risation g\'{e}om\'{e}trique des
              vari\'{e}t\'{e}s alg\'{e}briques affines},
   JOURNAL = {M\'{e}m. Soc. Math. France (N.S.)},
  FJOURNAL = {M\'{e}moires de la Soci\'{e}t\'{e} Math\'{e}matique de France. Nouvelle S\'{e}rie},
    NUMBER = {19},
      YEAR = {1985},
     PAGES = {124},
      ISSN = {0037-9484},
   MRCLASS = {32H35 (32C10 32F05)},
  MRNUMBER = {813252},
MRREVIEWER = {G. M. Khenkin},
}

@article {demailly2004numerical,
    AUTHOR = {Demailly, J.-P. and Paun, M.},
     TITLE = {Numerical characterization of the {K}\"{a}hler cone of a compact
              {K}\"{a}hler manifold},
   JOURNAL = {Ann. of Math. (2)},
  FJOURNAL = {Annals of Mathematics. Second Series},
    VOLUME = {159},
      YEAR = {2004},
    NUMBER = {3},
     PAGES = {1247--1274},
      ISSN = {0003-486X},
   MRCLASS = {32J27 (32Q15)},
  MRNUMBER = {2113021},
MRREVIEWER = {Philippe P. Eyssidieux},
       DOI = {10.4007/annals.2004.159.1247},
       URL = {https://doi.org/10.4007/annals.2004.159.1247},
}

@article{demaillycomplex,
  title={Complex Analytic and Differential Geometry},
  author={Demailly, J.-P.},
  journal={available at \href{https://www-fourier.ujf-grenoble.fr/~demailly/manuscripts/agbook.pdf}{https://www-fourier.ujf-grenoble.fr/~demailly/manuscripts/agbook.pdf}},
  year={2012},
URL={https://www-fourier.ujf-grenoble.fr/~Demailly/manuscripts/agbook.pdf}
}

@incollection {dinew2012pluripotential,
    AUTHOR = {Dinew, S. and Ko{\l}odziej, S.},
     TITLE = {Pluripotential estimates on compact {H}ermitian manifolds},
 BOOKTITLE = {Advances in geometric analysis},
    SERIES = {Adv. Lect. Math. (ALM)},
    VOLUME = {21},
     PAGES = {69--86},
 PUBLISHER = {Int. Press, Somerville, MA},
      YEAR = {2012},
   MRCLASS = {32U05 (32U40 32W20 53C55)},
MRREVIEWER = {Slimane Benelkourchi},
}

@book {griffiths1978principles,
    AUTHOR = {Griffiths, P. and Harris, J.},
     TITLE = {Principles of algebraic geometry},
      NOTE = {Pure and Applied Mathematics},
 PUBLISHER = {Wiley-Interscience [John Wiley \& Sons], New York},
      YEAR = {1978},
     PAGES = {xii+813},
      ISBN = {0-471-32792-1},
   MRCLASS = {14-01},
  MRNUMBER = {507725},
MRREVIEWER = {Gerhard Pfister},
}

@article {guan2010complex,
    AUTHOR = {Guan, B. and Li, Q.},
     TITLE = {Complex {M}onge-{A}mp\`ere equations and totally real
              submanifolds},
   JOURNAL = {Adv. Math.},
  FJOURNAL = {Advances in Mathematics},
    VOLUME = {225},
      YEAR = {2010},
    NUMBER = {3},
     PAGES = {1185--1223},
      ISSN = {0001-8708},
   MRCLASS = {32W20 (32Q15 53C55)},
  MRNUMBER = {2673728},
MRREVIEWER = {S\l awomir Dinew},
       DOI = {10.1016/j.aim.2010.03.019},
       URL = {https://doi.org/10.1016/j.aim.2010.03.019},
}

@book {guedj2017degenerate,
    AUTHOR = {Guedj, V. and Zeriahi, A.},
     TITLE = {Degenerate complex {M}onge-{A}mp\`ere equations},
    SERIES = {EMS Tracts in Mathematics},
    VOLUME = {26},
 PUBLISHER = {European Mathematical Society (EMS), Z\"{u}rich},
      YEAR = {2017},
     PAGES = {xxiv+472},
      ISBN = {978-3-03719-167-5},
   MRCLASS = {32W20 (32Q20 32U15 32U20 32U40 35J96)},
MRREVIEWER = {Slimane Benelkourchi},
       DOI = {10.4171/167},
       URL = {https://doi.org/10.4171/167},
}

@article {phong2010dirichlet,
    AUTHOR = {Phong, D. H. and Sturm, J.},
     TITLE = {The {D}irichlet problem for degenerate complex
              {M}onge-{A}mpere equations},
   JOURNAL = {Comm. Anal. Geom.},
  FJOURNAL = {Communications in Analysis and Geometry},
    VOLUME = {18},
      YEAR = {2010},
    NUMBER = {1},
     PAGES = {145--170},
      ISSN = {1019-8385},
   MRCLASS = {32W20 (35J96 35R01)},
  MRNUMBER = {2660461},
MRREVIEWER = {Yanir A. Rubinstein},
       DOI = {10.4310/CAG.2010.v18.n1.a6},
       URL = {https://doi.org/10.4310/CAG.2010.v18.n1.a6},
}

@article {to2018regularizing,
    AUTHOR = {Tô, T. D.},
     TITLE = {Regularizing properties of complex {M}onge-{A}mp\`ere flows
              {II}: {H}ermitian manifolds},
   JOURNAL = {Math. Ann.},
  FJOURNAL = {Mathematische Annalen},
    VOLUME = {372},
      YEAR = {2018},
    NUMBER = {1-2},
     PAGES = {699--741},
      ISSN = {0025-5831},
   MRCLASS = {32W20 (32U25 53C44)},
MRREVIEWER = {Ngoc Cuong Nguyen},
       DOI = {10.1007/s00208-017-1574-7},
       URL = {https://doi.org/10.1007/s00208-017-1574-7},
}

@article {tosatti2013chern,
    AUTHOR = {Tosatti, V. and Weinkove, B.},
     TITLE = {The {C}hern-{R}icci flow on complex surfaces},
   JOURNAL = {Compos. Math.},
  FJOURNAL = {Compositio Mathematica},
    VOLUME = {149},
      YEAR = {2013},
    NUMBER = {12},
     PAGES = {2101--2138},
      ISSN = {0010-437X},
   MRCLASS = {53C44 (32W20 53C55)},
MRREVIEWER = {Xiangwen Zhang},
       DOI = {10.1112/S0010437X13007471},
       URL = {https://doi.org/10.1112/S0010437X13007471},
}

@article {tosatti2015evolution,
    AUTHOR = {Tosatti, V. and Weinkove, B.},
     TITLE = {On the evolution of a {H}ermitian metric by its
              {C}hern-{R}icci form},
   JOURNAL = {J. Differential Geom.},
  FJOURNAL = {Journal of Differential Geometry},
    VOLUME = {99},
      YEAR = {2015},
    NUMBER = {1},
     PAGES = {125--163},
      ISSN = {0022-040X},
   MRCLASS = {53C44 (53C55)},
MRREVIEWER = {Chengjie Yu},
       URL = {http://projecteuclid.org/euclid.jdg/1418345539},
}

@article {tosatti2018kawa,
    AUTHOR = {Tosatti, V.},
     TITLE = {K{AWA} lecture notes on the {K}\"{a}hler-{R}icci flow},
   JOURNAL = {Ann. Fac. Sci. Toulouse Math. (6)},
  FJOURNAL = {Annales de la Facult\'{e} des Sciences de Toulouse. Math\'{e}matiques.
              S\'{e}rie 6},
    VOLUME = {27},
      YEAR = {2018},
    NUMBER = {2},
     PAGES = {285--376},
      ISSN = {0240-2963},
   MRCLASS = {53C44 (53C55)},
  MRNUMBER = {3831026},
       DOI = {10.5802/afst.1571},
       URL = {https://doi.org/10.5802/afst.1571},
}

@article{angella2022plurisigned,
  author={Angella, D. and Guedj, V. and Lu, C. H},
TITLE = {Plurisigned hermitian metrics},
   JOURNAL = {Trans. Amer. Math. Soc.},
  FJOURNAL = {Transactions of the American Mathematical Society},
    VOLUME = {376},
      YEAR = {2023},
    NUMBER = {7},
     PAGES = {4631--4659},
      ISSN = {0002-9947,1088-6850},
   MRCLASS = {32W20 (32Q15 32U05 35A23)},
  MRNUMBER = {4608427},
       DOI = {10.1090/tran/8916},
       URL = {https://doi.org/10.1090/tran/8916},
}

@article {guedj2022quasi2,
    AUTHOR = {Guedj, V. and Lu, C. H.},
     TITLE = {Quasi-plurisubharmonic envelopes 2: {B}ounds on
              {M}onge-{A}mp\`ere volumes},
   JOURNAL = {Algebr. Geom.},
  FJOURNAL = {Algebraic Geometry},
    VOLUME = {9},
      YEAR = {2022},
    NUMBER = {6},
     PAGES = {688--713},
      ISSN = {2313-1691},
   MRCLASS = {32W20 (32J18 32U05 35A23 35J96)},
  MRNUMBER = {4518244},
       DOI = {10.14231/ag-2022-021},
       URL = {https://doi.org/10.14231/ag-2022-021},
}

@article {popovici2016sufficients,
    AUTHOR = {Popovici, D.},
     TITLE = {Sufficient bigness criterion for differences of two nef
              classes},
   JOURNAL = {Math. Ann.},
  FJOURNAL = {Mathematische Annalen},
    VOLUME = {364},
      YEAR = {2016},
    NUMBER = {1-2},
     PAGES = {649--655},
      ISSN = {0025-5831},
   MRCLASS = {32J25 (32W20)},
  MRNUMBER = {3451400},
MRREVIEWER = {Jian Xiao},
       DOI = {10.1007/s00208-015-1230-z},
       URL = {https://doi.org/10.1007/s00208-015-1230-z},
}

@article {lamari1999courants,
    AUTHOR = {Lamari, A.},
     TITLE = {Courants k\"{a}hl\'{e}riens et surfaces compactes},
   JOURNAL = {Ann. Inst. Fourier (Grenoble)},
  FJOURNAL = {Universit\'{e} de Grenoble. Annales de l'Institut Fourier},
    VOLUME = {49},
      YEAR = {1999},
    NUMBER = {1},
     PAGES = {vii, x, 263--285},
      ISSN = {0373-0956},
   MRCLASS = {32J27 (32C30 32J15 32Q15)},
  MRNUMBER = {1688140},
MRREVIEWER = {Thierry Bouche},
       URL = {http://www.numdam.org/item?id=AIF_1999__49_1_263_0},
}

@article {Buchdahl1999compact,
    AUTHOR = {Buchdahl, N.},
     TITLE = {On compact {K}\"{a}hler surfaces},
   JOURNAL = {Ann. Inst. Fourier (Grenoble)},
  FJOURNAL = {Universit\'{e} de Grenoble. Annales de l'Institut Fourier},
    VOLUME = {49},
      YEAR = {1999},
    NUMBER = {1},
     PAGES = {vii, xi, 287--302},
      ISSN = {0373-0956},
   MRCLASS = {32J27 (32J15)},
  MRNUMBER = {1688136},
MRREVIEWER = {Thierry Bouche},
       URL = {http://www.numdam.org/item?id=AIF_1999__49_1_287_0},
}

@article {tosatti2022chern,
    AUTHOR = {Tosatti, V. and Weinkove, B.},
     TITLE = {The {C}hern-{R}icci flow},
   JOURNAL = {Atti Accad. Naz. Lincei Rend. Lincei Mat. Appl.},
  FJOURNAL = {Atti della Accademia Nazionale dei Lincei. Rendiconti Lincei.
              Matematica e Applicazioni},
    VOLUME = {33},
      YEAR = {2022},
    NUMBER = {1},
     PAGES = {73--107},
      ISSN = {1120-6330},
   MRCLASS = {53E30 (32J15 32W20)},
  MRNUMBER = {4444044},
MRREVIEWER = {Chandan Kumar Mondal},
       DOI = {10.4171/rlm/965},
       URL = {https://doi.org/10.4171/rlm/965},
}

@article {chiose2016kahler,
    AUTHOR = {Chiose, I.},
     TITLE = {The {K}\"{a}hler rank of compact complex manifolds},
   JOURNAL = {J. Geom. Anal.},
  FJOURNAL = {Journal of Geometric Analysis},
    VOLUME = {26},
      YEAR = {2016},
    NUMBER = {1},
     PAGES = {603--615},
      ISSN = {1050-6926},
   MRCLASS = {32J27 (32J15)},
  MRNUMBER = {3441529},
MRREVIEWER = {Florian Schrack},
       DOI = {10.1007/s12220-015-9564-z},
       URL = {https://doi.org/10.1007/s12220-015-9564-z},
}

@article {nie2017weak,
    AUTHOR = {Nie, X.},
     TITLE = {Weak solution of the {C}hern-{R}icci flow on compact complex
              surfaces},
   JOURNAL = {Math. Res. Lett.},
  FJOURNAL = {Mathematical Research Letters},
    VOLUME = {24},
      YEAR = {2017},
    NUMBER = {6},
     PAGES = {1819--1844},
      ISSN = {1073-2780},
   MRCLASS = {53C44 (32Q15 32W20 53C55)},
  MRNUMBER = {3762697},
MRREVIEWER = {Valentino Tosatti},
       DOI = {10.4310/MRL.2017.v24.n6.a13},
       URL = {https://doi.org/10.4310/MRL.2017.v24.n6.a13},
}

@article {hironaka77-actavietnam,
    AUTHOR = {Hironaka, H.},
     TITLE = {Bimeromorphic smoothing of a complex-analytic space},
   JOURNAL = {Acta Math. Vietnam.},
  FJOURNAL = {Acta Mathematica Vietnamica},
    VOLUME = {2},
      YEAR = {1977},
    NUMBER = {2},
     PAGES = {103--168},
      ISSN = {0251-4184},
   MRCLASS = {32C45},
  MRNUMBER = {499299},
MRREVIEWER = {Tom\'{a}s\ S\'{a}nchez-Giralda},
}

@inproceedings {wlodarczyk09-Resolution,
    AUTHOR = {W{\l}odarczyk, J.},
     TITLE = {Resolution of singularities of analytic spaces},
 BOOKTITLE = {Proceedings of {G}\"{o}kova {G}eometry-{T}opology {C}onference
              2008},
     PAGES = {31--63},
 PUBLISHER = {G\"{o}kova Geometry/Topology Conference (GGT), G\"{o}kova},
      YEAR = {2009},
      ISBN = {978-1-57146-136-0},
   MRCLASS = {32S45 (32C25)},
  MRNUMBER = {2500573},
}

@article {buchdahl00-NakaiMoishezon,
    AUTHOR = {Buchdahl, N.},
     TITLE = {A {N}akai-{M}oishezon criterion for non-{K}\"{a}hler surfaces},
   JOURNAL = {Ann. Inst. Fourier (Grenoble)},
  FJOURNAL = {Universit\'{e} de Grenoble. Annales de l'Institut Fourier},
    VOLUME = {50},
      YEAR = {2000},
    NUMBER = {5},
     PAGES = {1533--1538},
      ISSN = {0373-0956,1777-5310},
   MRCLASS = {32J15 (32C30)},
  MRNUMBER = {1800126},
MRREVIEWER = {Lin\ Weng},
       URL = {http://www.numdam.org/item?id=AIF_2000__50_5_1533_0},
}

@article{dang2023singularities,
    AUTHOR = {Dang, Q.-T.},
     TITLE = {{Singularities of the Chern-Ricci flow}},
   JOURNAL = {Anal. PDE},
  FJOURNAL = {Analysis \& PDE},
    VOLUME = {19},
      YEAR = {2026},
    NUMBER = {3},
     PAGES = {449--483},
}

@article {enders11-typeSingularities,
    AUTHOR = {Enders, J. and M\"{u}ller, R. and Topping, P. M.},
     TITLE = {On type-{I} singularities in {R}icci flow},
   JOURNAL = {Comm. Anal. Geom.},
  FJOURNAL = {Communications in Analysis and Geometry},
    VOLUME = {19},
      YEAR = {2011},
    NUMBER = {5},
     PAGES = {905--922},
      ISSN = {1019-8385,1944-9992},
   MRCLASS = {53C44},
  MRNUMBER = {2886712},
MRREVIEWER = {Juan-Ru\ Gu},
       DOI = {10.4310/CAG.2011.v19.n5.a4},
       URL = {https://doi.org/10.4310/CAG.2011.v19.n5.a4},
}

@article {nakamaye00-base-loci,
    AUTHOR = {Nakamaye, M.},
     TITLE = {Stable base loci of linear series},
   JOURNAL = {Math. Ann.},
  FJOURNAL = {Mathematische Annalen},
    VOLUME = {318},
      YEAR = {2000},
    NUMBER = {4},
     PAGES = {837--847},
      ISSN = {0025-5831,1432-1807},
   MRCLASS = {14C20 (14D05)},
  MRNUMBER = {1802513},
MRREVIEWER = {Thomas\ Bauer},
       DOI = {10.1007/s002080000149},
       URL = {https://doi.org/10.1007/s002080000149},
}

@article {Ein-Laz-Mus-Nak-Pop09-base,
    AUTHOR = {Ein, L. and Lazarsfeld, R. and Musta\c{t}\u{a},
              M. and Nakamaye, M. and Popa, M.},
     TITLE = {Restricted volumes and base loci of linear series},
   JOURNAL = {Amer. J. Math.},
  FJOURNAL = {American Journal of Mathematics},
    VOLUME = {131},
      YEAR = {2009},
    NUMBER = {3},
     PAGES = {607--651},
      ISSN = {0002-9327,1080-6377},
   MRCLASS = {14C20},
  MRNUMBER = {2530849},
MRREVIEWER = {Tomasz\ Szemberg},
       DOI = {10.1353/ajm.0.0054},
       URL = {https://doi.org/10.1353/ajm.0.0054},
}

@incollection {Nakamaye-book,
    AUTHOR = {Nakamaye, M.},
     TITLE = {Roth's theorem: an introduction to diophantine approximation},
 BOOKTITLE = {Rational points, rational curves, and entire holomorphic
              curves on projective varieties},
    SERIES = {Contemp. Math.},
    VOLUME = {654},
     PAGES = {75--108},
 PUBLISHER = {Amer. Math. Soc., Providence, RI},
      YEAR = {2015},
      ISBN = {978-1-4704-1458-0},
   MRCLASS = {11J04},
  MRNUMBER = {3477541},
MRREVIEWER = {Takao\ Komatsu},
       DOI = {10.1090/conm/654/13216},
       URL = {https://doi.org/10.1090/conm/654/13216},
}

@article {Birkar17-baselocus,
    AUTHOR = {Birkar, C.},
     TITLE = {The augmented base locus of real divisors over arbitrary
              fields},
   JOURNAL = {Math. Ann.},
  FJOURNAL = {Mathematische Annalen},
    VOLUME = {368},
      YEAR = {2017},
    NUMBER = {3-4},
     PAGES = {905--921},
      ISSN = {0025-5831,1432-1807},
   MRCLASS = {14C20 (14E30)},
  MRNUMBER = {3673639},
MRREVIEWER = {Chen\ Jiang},
       DOI = {10.1007/s00208-016-1441-y},
       URL = {https://doi.org/10.1007/s00208-016-1441-y},
}

@article {Cacciola-Lopez14-logcanonical,
    AUTHOR = {Cacciola, S. and Lopez, A. F.},
     TITLE = {Nakamaye's theorem on log canonical pairs},
   JOURNAL = {Ann. Inst. Fourier (Grenoble)},
  FJOURNAL = {Universit\'{e} de Grenoble. Annales de l'Institut Fourier},
    VOLUME = {64},
      YEAR = {2014},
    NUMBER = {6},
     PAGES = {2283--2298},
      ISSN = {0373-0956,1777-5310},
   MRCLASS = {14C20 (14B05 14E15 14F18)},
  MRNUMBER = {3331167},
MRREVIEWER = {Eugenii\ Shustin},
       DOI = {10.5802/aif.2913},
       URL = {https://doi.org/10.5802/aif.2913},
}

@article {Casini-McKernan-Mutaca14-baselocus,
    AUTHOR = {Cascini, P. and McKernan, J. and Musta\c{t}\u{a},
              M.},
     TITLE = {The augmented base locus in positive characteristic},
   JOURNAL = {Proc. Edinb. Math. Soc. (2)},
  FJOURNAL = {Proceedings of the Edinburgh Mathematical Society. Series II},
    VOLUME = {57},
      YEAR = {2014},
    NUMBER = {1},
     PAGES = {79--87},
      ISSN = {0013-0915,1464-3839},
   MRCLASS = {14A15 (14C20)},
  MRNUMBER = {3165013},
MRREVIEWER = {Tomasz\ Szemberg},
       DOI = {10.1017/S0013091513000916},
       URL = {https://doi.org/10.1017/S0013091513000916},
}

@article {boucksom-cacciola-lopez14-baseloci,
    AUTHOR = {Boucksom, S. and Cacciola, S. and Lopez,
              A. F.},
     TITLE = {Augmented base loci and restricted volumes on normal
              varieties},
   JOURNAL = {Math. Z.},
  FJOURNAL = {Mathematische Zeitschrift},
    VOLUME = {278},
      YEAR = {2014},
    NUMBER = {3-4},
     PAGES = {979--985},
      ISSN = {0025-5874,1432-1823},
   MRCLASS = {14C20 (14G17 14J40)},
  MRNUMBER = {3278900},
MRREVIEWER = {Yoshiaki\ Fukuma},
       DOI = {10.1007/s00209-014-1341-3},
       URL = {https://doi.org/10.1007/s00209-014-1341-3},
}

@incollection {tosatti18-nakamaye,
    AUTHOR = {Tosatti, V.},
     TITLE = {Nakamaye's theorem on complex manifolds},
 BOOKTITLE = {Algebraic geometry: {S}alt {L}ake {C}ity 2015},
    SERIES = {Proc. Sympos. Pure Math.},
    VOLUME = {97.1},
     PAGES = {633--655},
 PUBLISHER = {Amer. Math. Soc., Providence, RI},
      YEAR = {2018},
      ISBN = {978-1-4704-3577-6},
   MRCLASS = {32J25 (14C20 32Q15)},
  MRNUMBER = {3821165},
MRREVIEWER = {Simone\ Diverio},
       DOI = {10.1090/pspum/097.1/22},
       URL = {https://doi.org/10.1090/pspum/097.1/22},
}

@article {gill-smith14-chernricci,
    AUTHOR = {Gill, M. and Smith, D.},
     TITLE = {The behavior of the {C}hern scalar curvature under the
              {C}hern-{R}icci flow},
   JOURNAL = {Proc. Amer. Math. Soc.},
  FJOURNAL = {Proceedings of the American Mathematical Society},
    VOLUME = {143},
      YEAR = {2015},
    NUMBER = {11},
     PAGES = {4875--4883},
      ISSN = {0002-9939,1088-6826},
   MRCLASS = {53C44 (53C55)},
  MRNUMBER = {3391045},
MRREVIEWER = {Hassan\ Jolany},
       DOI = {10.1090/proc/12745},
       URL = {https://doi.org/10.1090/proc/12745},
}

@article{chiose2016invariance,
  AUTHOR = {Chiose, I.},
     TITLE = {On the invariance of the total {M}onge-{A}mp\`ere volume of
              {H}ermitian metrics},
   JOURNAL = {Ann. Fac. Sci. Toulouse Math. (6)},
  FJOURNAL = {Annales de la Facult\'e{} des Sciences de Toulouse.
              Math\'ematiques. S\'erie 6},
    VOLUME = {33},
      YEAR = {2024},
    NUMBER = {3},
     PAGES = {575--579},
      ISSN = {0240-2963,2258-7519},
   MRCLASS = {32W20},
  MRNUMBER = {4822910},
MRREVIEWER = {Masaya\ Kawamura},
       DOI = {10.5802/afst.1781},
       URL = {https://doi.org/10.5802/afst.1781},
}

@article {sherman-weinkove13-estimates,
    AUTHOR = {Sherman, M. and Weinkove, B.},
     TITLE = {Local {C}alabi and curvature estimates for the {C}hern-{R}icci
              flow},
   JOURNAL = {New York J. Math.},
  FJOURNAL = {New York Journal of Mathematics},
    VOLUME = {19},
      YEAR = {2013},
     PAGES = {565--582},
      ISSN = {1076-9803},
   MRCLASS = {53C44 (53C55)},
  MRNUMBER = {3119098},
MRREVIEWER = {Yanir\ A.\ Rubinstein},
       URL = {http://nyjm.albany.edu:8000/j/2013/19_565.html},
}

@article {Zhang10-scalar,
    AUTHOR = {Zhang, Z.},
     TITLE = {Scalar curvature behavior for finite-time singularity of
              {K}\"{a}hler-{R}icci flow},
   JOURNAL = {Michigan Math. J.},
  FJOURNAL = {Michigan Mathematical Journal},
    VOLUME = {59},
      YEAR = {2010},
    NUMBER = {2},
     PAGES = {419--433},
      ISSN = {0026-2285,1945-2365},
   MRCLASS = {53C44 (32W20)},
  MRNUMBER = {2677630},
MRREVIEWER = {Julien\ Keller},
       DOI = {10.1307/mmj/1281531465},
       URL = {https://doi.org/10.1307/mmj/1281531465},
}

@article {collins-tosatti-singular,
    AUTHOR = {Collins, T. C. and Tosatti, V.},
     TITLE = {A singular {D}emailly-{P}\u{a}un theorem},
   JOURNAL = {C. R. Math. Acad. Sci. Paris},
  FJOURNAL = {Comptes Rendus Math\'{e}matique. Acad\'{e}mie des Sciences.
              Paris},
    VOLUME = {354},
      YEAR = {2016},
    NUMBER = {1},
     PAGES = {91--95},
      ISSN = {1631-073X,1778-3569},
   MRCLASS = {32J25},
  MRNUMBER = {3439731},
MRREVIEWER = {Tsz\ On Mario Chan},
       DOI = {10.1016/j.crma.2015.10.012},
       URL = {https://doi.org/10.1016/j.crma.2015.10.012},
}

@article {FT11-astheno,
    AUTHOR = {Fino, A. and Tomassini, A.},
     TITLE = {On astheno-{K}\"{a}hler metrics},
   JOURNAL = {J. Lond. Math. Soc. (2)},
  FJOURNAL = {Journal of the London Mathematical Society. Second Series},
    VOLUME = {83},
      YEAR = {2011},
    NUMBER = {2},
     PAGES = {290--308},
      ISSN = {0024-6107,1469-7750},
   MRCLASS = {53C55 (53C80)},
  MRNUMBER = {2776638},
MRREVIEWER = {S\"{o}nke\ Rollenske},
       DOI = {10.1112/jlms/jdq066},
       URL = {https://doi.org/10.1112/jlms/jdq066},
}

@article{darvas2023transcendental,
  AUTHOR = {Darvas, T. and Reboulet, R. and Witt{ }Nystr\"om, D.
              and Xia, M. and Zhang, K.},
     TITLE = {Transcendental {O}kounkov bodies},
   JOURNAL = {J. Differential Geom.},
  FJOURNAL = {Journal of Differential Geometry},
    VOLUME = {132},
      YEAR = {2026},
    NUMBER = {1},
     PAGES = {135--178},
      ISSN = {0022-040X,1945-743X},
   MRCLASS = {14C20 (32C99 32Q15 53C55)},
  MRNUMBER = {5007919},
       DOI = {10.4310/jdg/1766433802},
       URL = {https://doi.org/10.4310/jdg/1766433802},
}

@article{das-hacon-paun2022mmp,
  AUTHOR = {Das, O. and Hacon, C. and P{\u a}un, Mi.},
     TITLE = {On the 4-dimensional minimal model program for {K}\"ahler
              varieties},
   JOURNAL = {Adv. Math.},
  FJOURNAL = {Advances in Mathematics},
    VOLUME = {443},
      YEAR = {2024},
     PAGES = {Paper No. 109615, 68},
      ISSN = {0001-8708,1090-2082},
   MRCLASS = {14E05 (14E30 32Q57)},
  MRNUMBER = {4719824},
       DOI = {10.1016/j.aim.2024.109615},
       URL = {https://doi.org/10.1016/j.aim.2024.109615},
}
\end{document}